\documentclass{elsarticle}

\usepackage{amsthm}
\usepackage{amsmath}
\usepackage{amssymb}
\usepackage{verbatim}
\usepackage{graphicx}
\usepackage{hyperref}
\usepackage{color}
\usepackage{mathrsfs}

\newcommand{\mcs}{\mathcal{S}}

\newcommand{\psd}{\succeq}

\newcommand{\la}{\langle}
\newcommand{\ra}{\rangle}
\newcommand{\ext}{{\rm ext \hspace*{0.5mm}}}

\newcommand{\Ker}{{\rm Ker \hspace*{0.5mm}}}
\newcommand{\rankspace}{{\rm rank \hspace*{0.5mm}}}
\newcommand{\sfT}{{\sf T}}

\newcommand{\lan}{\langle}

\newcommand{\GG}{{\mathcal G}}

\newcommand{\PP}{{\mathcal P}}
\newcommand{\HH}{{\mathcal H}}

\newcommand{\oR}{{\mathbb R}}

\newcommand{\EE}{\mathcal E}
\newcommand{\SSS}{{\mathcal S}}

\newcommand{\gd}{\text{\rm gd}}
\newcommand{\egd}{\text{\rm egd}}

\newcommand{\rank}{{\rm rank \hspace*{0.5mm}}}

\newcommand{\tw}{{\rm tw}}

\newcommand{\hG}{\widehat G}

\newcommand{\ignore}[1]{}

\newcommand{\NP}{\mathsf{NP}}

\newcommand{\conv}{\text{\rm conv}}

\newcommand{\sdp}{\text{\rm sdp}}

\newcommand{\ka}{\kappa}

\newcommand{\sla}{{\rm la_{\boxtimes}}}
\newcommand{\stp}{\boxtimes}
\newcommand{\lda}{{\rm la_{\square}}}
\newcommand{\MM}{{\mathcal M}}
\newcommand{\FF}{{\mathcal F}}
\newcommand{\face}{{F}}
\newcommand{\UU}{{\mathcal U}}
\newcommand{\fib}{{\rm fib}}
\newcommand{\WW}{{\mathcal W}}
\newcommand{\bc}{\begin{center}}
\newcommand{\ec}{\end{center}}
\newcommand{\GGt}{{\mathscr G}_2}

\newcommand{\CC}{{\mathcal C}}

\newcommand{\FFKc}{{\mathcal F}(F_3,K_4)}

\newcommand{\relint}{{\rm relint \hspace*{0.5mm}}}
\newtheorem{theorem}{Theorem}[section]
\newtheorem{corollary}[theorem]{Corollary}
\newtheorem{definition}[theorem]{Definition}
\newtheorem{proposition}[theorem]{Proposition}
\newtheorem{lemma}[theorem]{Lemma}

\newtheorem{example}{Example}[section]

\newtheorem*{theorem*}{Theorem}

\begin{document}
\newpage
\title{Forbidden minor  characterizations for   low-rank optimal solutions to semidefinite programs over the elliptope} 

\author[cwi]{M. E.-Nagy}
\ead{M.E.Nagy@cwi.nl}

\author[cwi,til]{M.~Laurent}
\ead{M.Laurent@cwi.nl}

\author[cwi]{A.~Varvitsiotis\corref{cor1}}
\ead{A.Varvitsiotis@cwi.nl}

\cortext[cor1]{Corresponding author: CWI, Postbus 94079,
	      1090 GB Amsterdam. Tel: +31 20 5924170; Fax: +31 20 5924199.}

\address[cwi]{Centrum Wiskunde \& Informatica (CWI), Science Park 123,
	        1098 XG Amsterdam,
	        The Netherlands.}
\address[til]{Tilburg University, 
P.O. Box 90153, 
5000 LE Tilburg, 
The Netherlands.}


\begin{abstract}
We study a new geometric graph parameter $\egd(G)$, defined as the smallest integer $r\ge 1$ for which any partial symmetric  matrix which is completable to a correlation matrix and  whose entries are specified at the positions of the edges of $G$, can be completed to a  matrix in the convex hull of correlation matrices of $\rank $ at most $r$.
This graph parameter is motivated by its relevance to the problem of finding low rank solutions to semidefinite programs over the elliptope, and also by its relevance to the bounded rank Grothendieck constant. Indeed,  $\egd(G)\le r$ if and only if the rank-$r$ Grothendieck constant of $G$ is equal to 1.
We show that the parameter $\egd(G)$ is minor monotone, we identify several classes of forbidden minors for $\egd(G)\le r$ and we give the full characterization for the case $r=2$. We also  show an upper bound for $\egd(G)$ in terms of a new tree-width-like parameter  $\sla(G)$, defined as the smallest $r$ for which $G$ is a minor of the strong product of a tree and $K_r$. We show that, for  any 2-connected graph $G\ne K_{3,3}$  on at least 6 nodes, $\egd(G)\le 2$ if and only if $\sla(G)\le 2$.
\end{abstract}

\begin{keyword}matrix completion, semidefinite programming, correlation matrix, Gram representation, graph minor, tree-width, Grothendieck constant. 
\end{keyword}

\maketitle

\section{Introduction}

\subsection{Semidefinite programs}
A semidefinite program (SDP) is a convex program defined as the minimization of a linear function over an affine section of the cone of positive semidefinite (psd)  matrices.
Semidefinite programming is a far reaching  generalization of  linear programing with a wide range of applications in a number of different  areas such as approximation algorithms \cite{MG12}, control theory \cite{PL03}, polynomial optimization \cite{L10} and quantum information theory \cite{B12}.
 A semidefinite  program in canonical primal form looks as follows:
\begin{equation}\tag{P}\label{sdpbasic}
\begin{aligned}
\inf   \  \la A_0,X\ra &  \\
 \text{subject to }  \la A_k,X\ra  & = b_k, \qquad  k=1,\ldots,m \\
 X & \psd  0,
\end{aligned}
\end{equation}
where $\la \cdot, \cdot \ra$ denotes the usual Frobenius  inner product of matrices and where $  A_k\  (0\le k \le m)  $ are $n$-by-$n$ symmetric matrices, called the  {\em coefficient} matrices of the SDP.
The generalized inequality $X \psd 0$ means that $X$ is positive semidefinite, i.e., all its eigenvalues are nonnegative. 

The field  of semidefinite programming has grown enormously in recent years. This success  can be attributed to the fact that SDP's  
have significant modeling power, exhibit a powerful duality theory  and  there exist efficient algorithms, both in theory and in practice, for solving them. 

 The first  landmark application of semidefinite programming is the work of Lov\'asz \cite{L79} on approximating the Shannon capacity of graphs   with the theta number, which   gives rise to   the only known  polynomial time algorithm for calculating these parameters in perfect graphs (see \cite{GLS}). Starting with the seminal work of Goemans and Williamson on the max-cut  problem \cite{GW95},   SDP's have also proven to be an invaluable tool in the design of  approximation algorithms for hard combinatorial optimization problems.  This  success is vividly illustrated by the fact that  many  SDP-based approximation algorithms are essentially optimal for a number of problems,  assuming  the validity of the Unique Games Conjecture~(see e.g. \cite{KKMD07,CMM06}).

\medskip
In this paper we consider the   problem of identifying conditions  that guarantee the   existence of low-rank optimal solutions for a certain  class of SDP's. Results of this type  are important  for approximation algorithms. Indeed, SDP's  are widely used as convex tractable relaxations for  hard combinatorial problems. Then, rank-one solutions typically correspond to  optimal solutions of the initial discrete problem and low-rank optimal solutions can  decrease the error of the rounding methods and   lead to improved performance guarantees. 

An illustrative example is  the max-cut problem where we are given as input an edge-weighted graph and the goal is to find a cut of maximum weight. It is  known that, using the Goemans and Williamson semidefinite programming relaxation, the max-cut problem can be approximated in polynomial time  within  a factor of 0.878   \cite{GW95}.
Furthermore, assuming that  this   SDP relaxation for max-cut has an optimal solution of rank 2 (resp., 3), this approximation ratio can be improved to 0.8844 (resp., 0.8818)~\cite{AZ05}. 

 Low-rank solutions to SDP's are also relevant  to the study of geometric representations of graphs. In this setting we consider representations  obtained by assigning vectors to the  vertices of a graph, where we impose  restrictions on the vectors labeling  adjacent vertices  (e.g. orthogonality, or  unit distance conditions). Then,  questions related  to the existence of low-dimensional representations  can be reformulated as the problem of deciding the existence of a low-rank solution to an appropriate SDP,  and they are connected to interesting graph properties 
 (see \cite{L09} for an overview).

\subsection{The Gram and extreme Gram dimension parameters}

Our main goal in  this paper is to  identify combinatorial conditions  guaranteeing  the existence of  low-rank optimal solutions to   SDP's.  This  question has been  raised by  Lov\'asz in  \cite{L98}. Quoting  Lov\'asz \cite[Problem 8.1]{L98} it is important to ``find combinatorial conditions that guarantee that the semidefinite relaxation has a solution of rank 1''.  Furthermore, the version of this problem ``with low rank instead of rank 1, also seems very interesting''.

Our   focus   lies  on  combinatorial conditions that capture  the sparsity  of the coefficient matrices of  a semidefinite program.  To encode this information, with  any semidefinite program of the form  \eqref{sdpbasic} we associate a graph $\mathcal{A}_P=~(V_P,E_P)$, called the {\em aggregate sparsity pattern} of \eqref{sdpbasic}, where $V_P=\{1,\ldots,n\}$ and 
$ij \in  E_P$  if and only if there exists   $k\in \{0,1,\ldots,m\}$ such that  $(\mathcal{A}_k)_{ij}\ne~0.$ 

The structure of the aggregate sparsity pattern of a  semidefinite program  can be used to prove the existence of low-rank optimal solutions. 
 This statement can be formalized by using   the following graph parameter, introduced in \cite{LV12}.

\begin{definition}\label{def:gram}\cite{LV12} The {\em Gram dimension} of a graph $G$ is defined as the smallest integer $r\ge 1$ with the following property:  Any  semidefinite program that attains its optimum and whose aggregate sparsity pattern is a subgraph of  $G$  has an optimal solution of rank at most $r$.   
\end{definition}

It is shown in~\cite{LV12} that the graph parameter $\gd(G)$ is minor monotone. Consequently, by the celebrated graph minor theorem of Robertson and Seymour \cite{RS04},   for any fixed integer $r\ge 1$, the graphs satisfying $\gd(G)\le r$ can be characterized by a finite list of minimal forbidden minors.  
The forbidden   minors  for  the graphs with $\gd(G)\le r$   for the values  $r=2,3$ and 4 were identified in \cite{LV12}. 
\begin{theorem}\label{thm:gd}\cite{LV12}
For any graph $G$ we have that
 \begin{itemize}
\item[(i)] $\gd(G)\le 2\ $ if and only if   $G$ has no $K_3$-minor,
 \item[(ii)] $\gd(G)\le 3 \ $  if and only if  $G$ has no $K_4$-minor,
 \item[(iii)] $\gd(G)\le 4 \ $  if and only if $G$   has no $K_5$ and $K_{2,2,2}$-minors.
 \end{itemize}
 \end{theorem}
Moreover it is  shown  in \cite{LV12} that there are  close  connections between   the Gram dimension and results concerning  Euclidean graph realizations  of Belk and Connelly \cite{Belk, BC07} and with linear algebraic properties of positive semidefinite  matrices, whose zero pattern is prescribed  by a fixed graph \cite{H96}.


\medskip
 In this paper we restrict our attention  to   SDP's  involving only constraints on the diagonal entries, namely, requiring  that every feasible matrix has all  its diagonal entries equal to 1. 
 Specifically, for  a graph $G=([n],E)$ and a vector of edge-weights $w \in \oR^E$, we consider SDP's of the following form: 
\begin{equation}\tag{$P^w_G$}\label{p2}
\sdp(G,w)=\max \sum_{ij \in E}w_{ij}X_{ij} \ \text{  s.t.  } \  X_{ii}=1\ (i\in [n]),\ X \psd 0.
\end{equation}

Semidefinite programs of this form arise naturally in the context of approximation algorithms. As an example, the Goemans-Williamson SDP relaxation of the max-cut problem (when formulated as a linear program in $\pm 1$ variables) fits  into this framework.  

Clearly,  for any $w \in\oR^E$, the optimal value of \eqref{p2} is attained since the objective function is linear and   the feasible region is a compact set. Moreover,  the aggregate sparsity pattern of \eqref{p2} is a subgraph  of  $G$. Consequently, 
 the semidefinite program   \eqref{p2} has an optimal solution of rank at most~$\gd(G)$ (recall Definition \ref{def:gram}).
Our objective in this paper is to strengthen the upper bound $\gd(G)$.  For this we introduce the following graph parameter.

\begin{definition}
 The {\em extreme Gram dimension}  of a graph $G=([n],E)$, denoted  $\egd(G)$,  is  defined as the smallest integer $r\ge 1$ such that,  for any $w \in \oR^E$, the program \eqref{p2} has an optimal solution of rank at most $r$. 
 \end{definition}
 
It follows from the definitions that $\egd(G)$ is upper bounded by $\gd(G)$, i.e., $$\egd(G)\le \gd(G).$$
Moreover, this inequality is strict, for instance, for the complete graph $K_n$. Indeed, as we will see in Section \ref{secbound},
\begin{equation*}\label{relKn}
\egd(K_n)=\left\lfloor {\sqrt{8n+1}-1\over 2}\right\rfloor <n=\gd(K_n), \text{ for } n \ge 2.
\end{equation*}

 
We will show that the graph parameter $\egd(\cdot)$ is minor monotone (Lemma~\ref{lemminor}).   Consequently, by the celebrated graph minor theorem of Robertson and Seymour, for any fixed integer $r \ge 1$, the class of graphs satisfying $\egd(G)\le r$ can be characterized by a finite list of minimal forbidden minors. It is known that a graph $G$ has $\egd(G)\le1$ if and only if $G$ has no $K_3$-minor \cite{Lau97}. Our main result in this paper is the  characterization of the graphs satisfying $\egd(G)\le 2$ in terms of two forbidden minors (cf. Theorem \ref{thm:main}).

 \medskip 
 
Next, we give a series of reformulations for the extreme Gram dimension that will be useful throughout the paper.  We start  by introducing  some necessary definitions and relevant  notation. Throughout, $\SSS^n$  denotes   the set of $n\times n$ symmetric matrices,  $\SSS^n_+$ is the cone of positive semidefinite (psd)  matrices and $\SSS^n_{++}$ is the cone of positive definite matrices.  A psd matrix whose diagonal entries are all equal to one  is called a {\em correlation matrix}.
The set 
 $$\EE_n=\{X\in \SSS^n_+\ :\  X_{ii}=1\ (i\in [n])\}$$
of all $n \times n$ correlation matrices   is known as  the {\em elliptope}. 
 For an integer $r\ge 1$, we define also 
the (in general non-convex) {\em bounded rank elliptope}   $$\EE_{n,r}=\{X\in \EE_n\ :\  \rankspace X\le r\}.$$
Given a graph $G=([n],E)$,   $\pi_E$  denotes  the projection from $\SSS^n$ onto the subspace $\oR^E$ indexed by the edge set of $G$, i.e., 
\begin{equation*}\label{projections}
\pi_E : \mcs^n \rightarrow \oR^E \qquad X \mapsto (X_{ij})_{ij \in E}.
\end{equation*}
Lastly, the {\em elliptope} of  the  graph $G=([n],E)$  is defined as    the projection of the elliptope $\EE_n$ onto the subspace indexed by the edge set of $G$:
$$\EE(G)=\pi_E(\EE_n).$$

The study of the elliptope  is motivated by its relevance  to the positive semidefinite matrix completion problem. Indeed, the  elements of $\EE(G)$ can be seen as the $G$-partial matrices that admit a completion to a full correlation matrix. A $G$-partial matrix is  a matrix whose entries are specified only at the off-diagonal positions corresponding to edges of $G$ and at the diagonal positions, with all diagonal entries being equal to 1. Consequently, deciding whether such a $G$-partial matrix admits a positive semidefinite  completion is equivalent to deciding membership in the elliptope $\EE(G)$.

\medskip
We now give the first  reformulation  for the extreme Gram dimension. For a graph $G=([n],E)$ and  $w\in \oR^E$, consider the rank-constrained~SDP:
\begin{equation}\label{pythasvlbcduw8er}
\sdp_r(G,w)=\max \sum_{ij \in E} w_{ij}X_{ij} \ \text{ s.t. } \ X \in \EE_{n,r}.
\end{equation}
 Then, it is clear that $\egd(G)$  can be equivalently defined   as the smallest integer $r\ge 1$ for which equality holds:
\begin{equation}\label{oirghncssoeuf}
\sdp(G,w)=\sdp_r(G,w), \  \text{ for all } w \in \oR^E.
\end{equation}

\medskip
For the second reformulation of the parameter $\egd(\cdot)$, we first observe    that program \eqref{p2}  is equivalent~to
\begin{equation}\label{p2b}
\sdp(G,w)=\max\  w^\sfT x \ \text{ s.t. } \ x \in \EE(G),
\end{equation}
and thus it corresponds to    optimization over the projected elliptope $\EE(G)$.
On the other hand, as  its objective function 
 is linear,    the (non-convex rank constrained)  program~\eqref{pythasvlbcduw8er}  can be equivalently reformulated as  optimization over the convex set $\pi_E(\conv (\EE_{n,r})$). That is, 
\begin{equation}\label{rel1bis}
 \sdp_r(G,w)=\max \ w^\sfT x \ \text{ s.t. }\ x \in \pi_E(\conv (\EE_{n,r})).
 \end{equation}
   Then, in view of \eqref{oirghncssoeuf},  we arrive at  the following geometric  reformulation for~$\egd(\cdot)$.

\begin{lemma}\label{def:egdgeom} The {\em extreme Gram dimension} of a graph $G=([n],E)$  is equal to the smallest integer $r\ge1$ for which  \index{dimension!extreme Gram}
\begin{equation}
\EE(G)=\pi_E(\conv(\EE_{n,r})).
\end{equation}
\end{lemma}

Using this geometric reformulation for the parameter $\egd(\cdot)$,  we are now in a position to   explain  why we have chosen to name it as the {\em extreme} Gram dimension. 
Since  the inclusion  $\pi_E(\conv(\EE_{n,r})) \subseteq \EE(G)$ is always valid, it follows from Lemma \ref{def:egdgeom} that  $\egd(G)$  is equal to the smallest $r\ge 1$ for which the reverse inclusion $\EE(G) \subseteq\pi_E(\conv(\EE_{n,r})) $ holds.  Moreover, since  $\EE(G)$ is a compact convex subset of $\oR^{E}$, by the Krein--Milman theorem, 
 $\EE(G)$ is equal to the convex hull of its set of extreme points. 
 With $\ext \EE(G)$ denoting the set of extreme points of $\EE(G)$, 
 it follows that 
 $$\EE(G)\subseteq \pi_E(\conv (\EE_{n,r})) \Longleftrightarrow \ext \EE(G)\subseteq \pi_E(\EE_{n,r}).$$
 Summarizing, the parameter  $\egd(G)$ can be reformulated as the smallest integer $r\ge 1$ for~which:
\begin{equation}\label{eq:inequality}
\ext \EE(G)\subseteq \pi_E(\EE_{n,r}).
\end{equation}
In other words, $\egd(G)$ is equal to the smallest $r\ge 1$ for which every extreme point of $\EE(G)$ has a positive semidefinite  completion of rank at most $r$.

\subsection{Relation with the Gram dimension}

As we will now see, both  $\gd(\cdot)$ and $\egd(\cdot)$ can be phrased within the common framework  of   Gram representations introduced below.  This reformulation will allow  us to clarify the relationship between these two parameters.

\begin{definition}\label{def:gramdimx} Given  a graph $G=([n],E)$ and 
a vector $x\in \oR^E$, a {\em Gram representation} of $x$ in $\oR^r$
is a set of unit vectors $p_1,\ldots,p_n\in \oR^r$  such that 
$$x_{ij}=p_i^\sfT p_j\  \   \forall  \{i,j\}\in  E.$$
The {\em Gram dimension} of $x \in \EE(G)$, denoted by $\gd(G,x)$, is the smallest integer $r\ge 1$ for which $x$ has such a Gram representation in $\oR^r$.
\end{definition}


Recall that, for a matrix $X\in \mathcal S^n$,   $X \in \EE_n$  if and only if there exists a family of unit vectors $p_1,\ldots,p_n$ such that $X_{ij}=p_i^\sfT p_j$ for all $i,j\in [n]$.  Hence, for a vector $x \in \EE(G)$, it is easy to see that $\gd(G,x)$ is equal to the smallest rank of a completion for $x$ to a full correlation matrix.  

Using this notion of Gram representations, we find the following equivalent definition for the Gram dimension of a graph, as originally introduced in \cite{LV12}. We sketch a proof of this fact  for clarity.

\begin{lemma}
For any graph $G=(V,E)$,
\begin{equation}\label{eq:gdrefor}
\gd(G)=\max_{x\in \EE(G)} \gd(G,x).
\end{equation}
\end{lemma}

\begin{proof}
Let $x\in \EE(G)$. Then $\gd(G,x)$ is the smallest $r$ for which the  SDP:
$$\min \ 0 \ \text{ s.t. } X\in \EE_n,\ X_{ij}=x_{ij}\ \forall ij\in E$$
has an optimal solution of rank at most $r$.
As the aggregate sparsity pattern of this SDP is equal to  $G$, it follows that $\gd(G)\ge \gd(G,x)$. This shows the inequality
$\gd(G)\ge \max_{x\in\EE(G)}\gd(G,x)$.
We now show the reverse inequality: $\gd(G)\le \max_{x\in\EE(G)}\gd(G,x)=:r.$
For this consider an SDP of the form (\ref{sdpbasic}) whose aggregate sparsity pattern is a subgraph of $G$.
Let $X$ be an optimal solution of (\ref{sdpbasic}); we construct another optimal solution $X'$ with rank at most $r$.
For simplicity let us assume that all diagonal entries of $X$ are positive (if not, just work with the principal submatrix of $X$ with only positive diagonal entries). With $D$ denoting the diagonal matrix with diagonal entries $\sqrt{X_{ii}}$, we can rescale $X$ so that $Y:= D^{-1}XD^{-1}$ belongs to $\EE_n$. Then, the projection $y=\pi_E(Y)$ belongs to $\EE(G)$. Hence there exists a matrix $Y'\in \EE_n$ of rank at most $r$ such that $\pi_E(Y')=y$. Scaling back, the matrix $X'=DY'D$ is a psd completion of $x$ and thus it is also an optimal solution of (\ref{sdpbasic}). Moreover, $X'$   has rank at most $r$, which concludes the proof.
\end{proof}

Moreover, 
we have  the following analogous reformulation for the extreme Gram dimension. 
\begin{lemma}\label{eq:reforegd}For any graph $G$, 
\begin{equation}\label{eq:egd}
  \egd(G)=\max_{x \in \ext\EE(G)}\gd(G,x).
  \end{equation}
\end{lemma}

\begin{proof}
Let $x\in \ext \EE(G)$. By \eqref{eq:inequality}, $x$ has a psd completion of rank at most $\egd(G)$ and thus $\gd(G,x)\le \egd(G)$.
This shows  $\max_{x\in \ext\EE(G)} \gd(G,x)\le \egd(G)$. The reverse inequality 
follows directly from~\eqref{eq:inequality}.
\end{proof}

Combining \eqref{eq:gdrefor} and \eqref{eq:egd} we find again the inequality: $\egd(G)\le \gd(G)$.

\subsection{Relation with the bounded rank Grothendieck constant} 

The study of  programs of the form \eqref{pythasvlbcduw8er} is of significant practical interest,   the main motivation  coming from statistical mechanics and in particular from the {\em $r$-vector model} introduced by Stanley \cite{S68}. This model consists of an interaction graph $G=(V,E)$, where vertices correspond to particles and edges indicate whether there is interaction (ferromagnetic or antiferromagnetic)  between the corresponding pair of particles. Additionally, there is a potential function $A:~V\times V\rightarrow  \oR$ satisfying $A_{ij}=0$ if $ij\not \in E$, $A_{ij}>0$ if there is ferromagnetic interaction between $i$ and $j$ and $A_{ij}<0$ if there is antiferromagnetic interaction between $i$ and $j$.  Additionally, particles possess a vector valued {\em spin} given by a function $f: V \rightarrow  \mathbb{S}^{r-1}$, where $ \mathbb{S}^{r-1}$ denotes the unit sphere in $\oR^r$. Assuming that there is no external field acting on the system, its total energy is given by the {\em Hamiltonian} defined as \index{Hamiltonian}
$$H(f)=-\sum_{ij\in E}A_{ij} f(i)^\sfT f(j).$$
A {\em ground state}  \index{ground state} is a configuration of spins that minimizes the Hamiltonian. The case $r=1$ corresponds to the {\em Ising model}, the case $r=2$ corresponds to the {\em XY model} and the case $r=3$ to the {\em Heisenberg model}.   Consequently, calculating the Hamiltonian and computing ground states in any of these models amounts to solving a rank-constrained semidefinite program of the form \eqref{pythasvlbcduw8er}.  

As the rank function is  non-convex and  non-differentiable, such problems are     computationally challenging.  Indeed, problem (\ref{pythasvlbcduw8er}) (or its reformulation (\ref{rel1bis})) is hard.  In  the case  $r=1$, the feasible region of \eqref{rel1bis} is equal to the cut polytope of the graph $G$  (in $\pm 1$ variables) and thus (\ref{rel1bis})   is $\NP$-hard. 
 It is believed that \eqref{pythasvlbcduw8er} is also  $\mathsf{NP}$-hard for any fixed  integer $r\ge 2$ (cf., e.g., the quote of Lov\'asz \cite[p. 61]{Lo01}). For any $r\ge 2$, it is shown in \cite{ELV12}  that membership in $\pi_E(\conv (\EE_{n,r}))$
is $\NP$-hard.  This  motivates the need for identifying tractable instances for~\eqref{pythasvlbcduw8er}.

\medskip
Clearly \eqref{p2} is a semidefinite programming  relaxation for program \eqref{pythasvlbcduw8er} obtained by removing the rank constraint. The quality of this relaxation is measured by its integrality gap defined below.

\begin{definition}The {\em  rank-$r$ Grothendieck constant} of a graph $G$, denoted as $\ka(r,G)$,  is defined as 
\begin{equation}
\ka(r,G)=\sup_{w \in \oR^E} \frac{\sdp(G,w)}{\sdp_r(G,w)}.
\end{equation}
\end{definition}


For $r=1$,  the special case where  $G$ is a complete bipartite graph  was studied by A. Grothendieck \cite{Gr}, although in a quite different language, and 
for general graphs by Alon et al. \cite{AMMN05}. The general case $r\ge 2$ is studied by Bri\"et et al. \cite{BOV10},   their main motivation  being   the polynomial-time approximation of ground states of spin glasses.


The extreme Gram dimension of a graph is closely related to the rank-r Grothendieck constant of a graph  as we now point out.
Indeed it follows directly  from the definitions that a graph has extreme Gram dimension at most $r$ if and only if its rank-$r$ Grothendieck constant is 1, i.e.,
$$\egd(G)\le r \Longleftrightarrow \kappa(r,G)=1.$$
Hence, for any  graph $G$  with  $\egd(G)\le r$,    we have that   $\sdp(G,w)=\sdp_r(G,w)$ for all $ w \in \oR^E$, and thus the value of program \eqref{pythasvlbcduw8er} can be approximated within arbitrary precision in polynomial time.

\subsection{Contributions and outline of the paper}
We now briefly summarize the main contributions of the paper. Our first result is to show that the new graph parameter $\egd(G)$ is minor monotone. As a consequence the class $\GG_r$ consisting of all  graphs $G$ with $\egd(G)\le r$ can be characterized by finitely many minimal forbidden minors. It is known that for the case $r=1$ the only forbidden minor is $K_3$, i.e., the class $\GG_1$ consists of all forests  \cite{Lau97}.  One of the main contributions of this paper is a complete characterization of the class $\GG_2$ (Theorem \ref{theomain}).

Furthermore, we identify three families of graphs $F_r, G_r$,  $H_r$ which are forbidden minors for the class $\GG_{r-1}$. This gives all the minimal forbidden minors for $r\le 2$. The graphs $G_r$ were already considered in \cite{V98,Ko01}.

On the other hand  we show an upper bound for the extreme Gram dimension  in terms of a tree-width-like parameter. 
This graph parameter, which we denote as $\sla(G)$, is defined as the smallest integer $r$ for which $G$ is a minor of the strong product $T\stp K_r$ of a tree $T$ and  the complete graph $K_r$. We  call it the {\em strong largeur d'arborescence} of $G$, in analogy with the {\em largeur d'arborescence} $\lda(G)$ introduced by Colin de Verdi\`ere \cite{V98}, defined similarly  by replacing the strong product with the  Cartesian product of graphs. Another main contribution is to show the upper bound: $\egd(G)\le \sla(G)$.

Our main result is that, for a graph $G\ne K_{3,3}$ which is 2-connected and has  at least 6 nodes,  $\egd(G)\le 2$ if and only if $\sla(G)\le 2$ if and only if $G$ does not have $F_3$ or $H_3$ as a minor.
We also characterize the graphs with $\sla(G)\le 2$ and recover the characterization of \cite{Ko00} for the graphs with $\lda(G)\le 2$.

The results and techniques in the paper come in two flavours: in Section \ref{sec_forbminor} they rely mostly on the geometry of faces of the elliptope and linear algebraic tools to construct suitable extreme points of the projected elliptope and, in Section \ref{secmainG2},  they are purely graph theoretic.

\medskip
The paper is organized as follows.
Section \ref{secprel} contains preliminaries about graphs 
and basic facts about the geometry of the faces of the elliptope. 
In Section \ref{secprop} we study properties of the new graph parameter $\egd(G)$. In particular, in Section \ref{secminor} we show minor-monotonicity and investigate the behaviour under the clique-sum graph operation  and in Section \ref{secbound} we show some bounds on $\egd(G)$.
In Section \ref{secupperbound}  we introduce the strong largeur d'arborescence parameter $\sla(G)$ and we show that it upper bounds the extreme Gram dimension, i.e.,  that  $\egd(G)\le \sla(G)$.
In Sections \ref{secFr}-\ref{secHr} 
we compute the extreme Gram dimension of the
three graph classes $F_r$, $G_r$ and $H_r$ 
and in  Section \ref{secK33} we compute the extreme Gram dimension of  the graphs $K_5$ and $K_{3,3}$, which play a special role within the class $\GG_2$.
Section \ref{secmainG2} is devoted to identifying the forbidden minors  for the class $\GG_2$. In  Section \ref{secchordal} we characterize the chordal graphs in $\GG_2$ (Theorem \ref{theofree}). In Section \ref{secnoK4} we show that any graph with no minor $F_3$ or $K_4$ admits a chordal extension avoiding these two minors (Theorem \ref{theonoF3K4}) and in Section \ref{secnoH3} we show the analogous result for graphs with no $F_3$ and $H_3$ minor (Theorem \ref{theonoF3H3}).
Finally in Section \ref{seclast}  we characterize the graphs with $\sla(G)\le 2$, we explain the links to results about $\lda(G)$, and we point out connections with the graph parameter $\nu(G)$ of Colin de Verdi\`ere \cite{V98}.

\section{Preliminaries}\label{secprel}
\subsection{Preliminaries about graphs}

We recall some definitions about graphs.
Let $G=(V,E)$ be  a  graph, we also denote its node set by $V(G)$ and its edge set by $E(G)$. A {\em component} is a maximal connected subgraph of $G$. 
A   {\em cutset}  is  a set $U\subseteq V$ for which $G\backslash U$ 
(deleting the nodes in $U$) has more connected components than $G$, $U$ is  a {\em cut node} if $|U|=1$, and  $G$ is {\em 2-connected} if it is connected and has no cut node.  
  For  $W\subseteq V$, $G[W]$ is the subgraph induced by $W$.
 Given  $\{u,v\}\not\in E(G)$,   $G+\{u,v\}$ is the graph obtained by adding the edge $\{u,v\}$ to $G$.

Given an edge $e=\{u,v\}\in E$, $G\backslash e=(V,E\setminus \{e\})$ is the graph obtained from $G$ by {\em deleting} the edge $e$ and  $G\slash e$ 
is obtained  by {\em contracting} the edge $e$: Replace  the two nodes $u$ and $v$ by a new node, adjacent to all the neighbors of $u$ and $v$. 
A graph $M$ is  a {\em minor} of  $G$, 
denoted as $M \preceq G$,  
if $M$ can be obtained from $G$  by a series of edge deletions and contractions and node deletions. 
Equivalently, 
$M$ is a minor of a connected graph $G$ if  there is a partition  of $V(G)$ into nonempty  subsets $\{ V_i: i\in V(M)\}$ where each $G[V_i]$ is connected and, for each edge 
 $\{i,j\}\in E(M)$, there exists at least one edge in $G$ between  $V_i$ and  $V_j$. 
 The collection $\{V_i: i\in V(M)\}$ is called an {\em $M$-partition} of $G$ and the $V_i$'s  are  the  {\em classes} of the partition. 

Given a finite list $\MM$ of graphs, $\FF(\MM)$ denotes the collection of all graphs that do not admit any  graph in $\MM$ as a minor. 
  By the celebrated graph minor theorem of Robertson and Seymour \cite{RS04}, any family  of graphs which is closed under the operation of  taking minors is of the form $\FF(\MM)$ for some finite set $\MM$ of graphs. In this setting, closed means that every minor of a graph in the family is also contained  in the family. 
 
 A {\em graph parameter} is any function from the set of graphs (up to isomorphism)  to the natural numbers.   A graph parameter $f(\cdot)$ is called   {\em minor monotone} if 
  $$f(G\backslash e)\le f(G) \text{ and }  f(G\slash e) \le f(G),$$ for any graph $G$ and any edge $e$ of $G$.  Given a minor monotone graph parameter $f(\cdot)$ and a  fixed integer $k\ge 1$, the family  of graphs $G$ satisfying  $f(G)\le k$ is closed under taking minors. Then, for any fixed integer $k\ge 1$ there exists a forbidden minor characterization for the family of graphs satisfying $f(G)\le k$.

A {\em homeomorph} (or subdivision) of a graph $M$  is obtained by replacing its edges  by paths.  When $M$ has maximum degree at most 3, $G$ admits $M$ as a minor if and only if it contains a homeomorph of $M$ as a subgraph.
 
A {\em clique} in $G$ is a set of pairwise adjacent nodes and $\omega(G)$ denotes the maximum cardinality of a clique in $G$.
A $k$-clique is a clique of cardinality~$k$.

Let  $G_1=(V_1,E_1)$, $G_2=(V_2,E_2)$ be two graphs,  where  $V_1\cap V_2$  is a clique in both $G_1$ and $G_2$. Their {\em clique sum} is the graph $G=(V_1\cup V_2, E_1\cup E_2)$, also called their 
{\em clique $k$-sum} when $k=|V_1\cap V_2|$.

If $C$ is a circuit in $G$, a {\em chord} of $C$ is an edge $\{u,v\}\in E$ where $u$ and $v$ are two nodes of $C$ that are not consecutive on  $C$.
$G$ is said to be {\em chordal} if every circuit of length at least 4 has a chord. As is well known, a graph $G$ is chordal if and only if $G$ is a clique sum of cliques.

The {\em Cartesian product} \index{graph!Cartesian product} of two graphs $G=(V,E)$ and $G'=(V',E')$, denoted by  $G\square G'$, \index{graph!strong product}
  is the graph 
with  node set $V\times V'$, where  distinct   nodes $(i,i'), (j,j')\in V\times V'$ are adjacent in $G\square G'$ when $i=j$ and $(i',j')\in G'$,  or $(i,j)\in G$ and $i'=j'$.

\newcommand{\tcolred}{\textcolor{red}}
\newcommand{\tcolgreen}{\textcolor{green}}
\newcommand{\tcolblue}{\textcolor{blue}}
\newcommand{\tcolmag}{\textcolor{magenta}}

The {\em strong product} of two graphs $G=(V,E)$ and $G'=(V',E')$, denoted by  $G\stp G'$, \index{graph!strong product}
  is the graph 
with  node set $V\times V'$, where distinct   nodes $(i,i'), (j,j')\in V\times V'$ are adjacent in $G\stp G'$ when
$i=j$ or $(i,j)\in E$, and $i'=j'$ or $(i',j')\in E'$.

The  {\em treewidth} of a graph $G$, denoted by $\tw(G)$,  
 is the smallest integer $k\ge 1 $ such that $G$ is contained in a clique sum of copies of $K_{k+1}$. This  parameter  was introduced by Robertson and Seymour in their fundamental work on graph minors  \cite{RS04} and is  commonly used  in the parameterized complexity analysis of graph algorithms. It is known that   $\tw(\cdot)$ is a minor-monotone graph parameter and that  $\tw(K_n)=n-1$ (see e.g. \cite{Diest}).
   It follows from the above definition that,   if  $G$ is obtained as the clique sum of $G_1$ and $G_2$ then 
\begin{equation}\label{tr:csum}
\tw(G)=\max\{ \tw(G_1),\tw(G_2)\}.
\end{equation}
 Colin de Verdi\`ere \cite{V98} introduced the following treewidth-like parameter: 
The {\em largeur d'arborescence} of a graph $G$, denoted by $ \lda(G)$,  is  the smallest  integer $r\ge 1$ for which $G$ is a minor of $T\Box K_r$ for some tree T.  Then,
$$\tw(G)\le \lda(G)\le \tw(G)+1,$$
 where the upper bound is shown in \cite{V98} and the lower bound in \cite{H96}.
 
The {\em strong largeur d'arborescence} of a graph $G$, denoted by $\sla(G)$, is  the smallest  integer $r\ge 1$ for which $G$ is a minor of $T\stp K_r$ for some tree T. 
In Section \ref{secupperbound} we will come back to this  parameter and we will show that it is  an upper bound for the extreme Gram dimension.

\subsection{Preliminaries about positive semidefinite matrices}\label{secmat}
Throughout we set $[n]=\{1,\dots,n\}$.  For   $U\subseteq [n]$ and $X \in \SSS^n$, $X[U]$ denotes the principal submatrix of $X$ with row and column indices in $U$ and,
 for $j\in [n]$, $X[\cdot, j]$ denotes the $j$-th column of $X$.
For a set $A\subseteq \oR^n$, $\lan A\ra$ denotes the vector space spanned by $A$ and  $\conv A$ denotes the convex hull of $A$. 

We group here some basic properties of    positive semidefinite matrices that we will use throughout. Given vectors $p_1,\ldots,p_n \in \oR^k$ ($k\ge 1$), we let  the matrix ${\rm Gram}(p_1,\ldots,p_n)=(p_i^\sfT p_j)_{i,j=1}^n$ denote  their  {\em Gram matrix}. 
Then, the rank of  ${\rm Gram}(p_1,\ldots,p_n)$  is equal to $ \dim \la p_1,\ldots,p_n\ra$.
 If $X=  {\rm Gram}(p_1,\ldots,p_n)$ we also  say that the $p_i$'s form  a  {\em Gram representation} of $X$.  As is well known, a  matrix $X\in \SSS^n$ is positive semidefinite if and only if 
\begin{equation}\label{gram}
 \exists p_1,\ldots,p_n\in \oR^k \text{ for some } k\ge 1  \text{ such that } X={\rm Gram}(p_1,\ldots,p_n).
\end{equation}
For $X\in \SSS^n$, $\Ker X$ is the kernel of $X$, consisting of the vectors $u\in \oR^n$ such that $Xu=0$.
When $X\succeq 0$, $u\in \Ker X$ (i.e., $Xu=0$) if and only if $u^TXu=0$.
Given  a matrix  $X\in \SSS^n$ in block-form
$$X=\left(\begin{matrix}A & B^T\cr B & C\end{matrix}\right),$$ the following holds: 
\begin{equation}\label{relker}
\text{if } X\succeq 0, \text{ then } \Ker A\subseteq \Ker B.
\end{equation}

\subsection{The geometry  of the elliptope }\label{sec:geom}


 In this section we group some geometric properties  of the elliptope, which we will need in the paper. First we recall some basic definitions and facts about faces of convex sets.

Given  a convex set $K$, a set $F\subseteq K$ is a {\em face} of $K$ if, for all $x\in F$, the condition $x=ty+(1-t)z$ with $y,z\in K$ and $t\in (0,1)$ implies $y,z\in F$. For $x\in K$ the smallest face $F(x)$ of $K$ containing $x$ is well defined, it is the unique face of $K$ containing $x$ in its relative interior.
A point $x\in K$ is an {\em extreme point} of $K$ if $F(x)=\{x\}.$ We denote the set of extreme points of a convex set $K$ by $\ext K$.
A useful property of extreme points is that if $F$ is a face of the convex set $K$~then 
\begin{equation}\label{extremepoints}
\ext F \subseteq \ext K.
\end{equation}
Moreover, $z$ is said to be a {\em perturbation} of $x\in K$ if $x\pm \epsilon z\in K$ for some $\epsilon>0$, then the segment $[x-\epsilon z,x+\epsilon z]$ is contained in $F(x)$ and the dimension of $F(x)$ is equal to the dimension of the linear space $\PP(x)$ of perturbations of $x$.



\medskip
We now   recall some facts about the faces of the elliptope that we need here. We refer e.g. to \cite{LP96} for details.
 For a matrix $X \in \EE_n$,    the smallest face $\face(X)$ of $\EE_n$  containing $X$ is given by
\begin{equation}\label{eq:ellrank}
\face(X)=\{Y\in \EE_n: \Ker X \subseteq \Ker Y\}.
\end{equation}
Therefore,    two matrices in the relative interior of a face $F$ of $\EE_n$ have the same rank, while 
 $\rank X > \rank Y$ if $X$ is in the relative interior of $F$ and $Y$ lies  on the boundary of $F$.
 Here is the explicit description of the space $\PP(X)$ of perturbations of a matrix $X\in \EE_n$.

  \begin{proposition}(\cite{LT}, see also \cite[\S 31.5]{DL97})\label{propface}
   Let $X\in \EE_n$ with rank $r$. Let     $u_1,\ldots, u_n\in \oR^r$ be a Gram representation of $X$, let   $U$   be the $r\times n$ matrix with columns $u_1,\ldots,u_n$ and set 
$\UU_V=\lan u_1u_1^\sfT,\ldots, u_nu_n^\sfT\ra\subseteq \SSS^r.$
    The space of perturbations $\PP(X)$ at $X$ is given by 
    \begin{equation}\label{relpert}
    \PP(X)=U^\sfT \UU_V^\perp U= \{U^\sfT RU: R\in \SSS^r, \lan R, u_iu_i^\sfT \ra =0 \ \forall i\in [n]\}
    \end{equation}
    and  the dimension of the smallest face $\face(X)$ of $\EE_n$ containing $X$ is 
\begin{equation}\label{eqdim}
\dim \face(X)=\dim \PP(X)={r+1\choose 2}-\dim \ \UU_V.
\end{equation}
In particular, $X$ is an extreme point of $\EE_n$ if and only if 
\begin{equation}
\label{eqdimextreme0}
{r+1\choose 2} =\dim \ \UU_V.
\end{equation}
Hence, if $X \in \ext\EE_n$ with $\rank X =r$ then 
\begin{equation}\label{eqdimextreme}
{r+1\choose 2}\le n.
\end{equation}
  \end{proposition}

An application for the previous proposition is the following example:

\begin{example}\label{example} Let $e_1,\dots,e_r\in\oR^r$ be the standard unit vectors.
The matrix with  Gram representation 
$\{e_i:\, i\in[r]\}\cup\{(e_i+e_j)/\sqrt{2}:\,1\le i<j\le r\}$
is an extreme point of $\EE_n$, since $\UU_V$ is full dimensional in $\SSS^r$, where $n=\binom{r+1}{2}$.
\end{example}

 The next  theorem   shows that every number in the range prescribed in~\eqref{eqdimextreme} corresponds to an extremal element of $\EE_n$.
  
\begin{theorem}\cite{LT}\label{dcfv} For any natural number $r$ satisfying  $\binom{r+1}{2}\le n$ there exists a matrix $X \in \EE_n$  which is an extreme point of $\EE_n$ and has  rank equal to $r$.
\end{theorem}



Next we establish  some tools which will be useful  to study the extreme points of the projected elliptope  $\EE(G)$. 

\begin{lemma}\label{lemextG1}
Consider a partial matrix  $x\in \EE(G)$ and let  $X\in \EE_n$ be a rank $r$ completion of $x$ with Gram representation $\{u_1,\ldots,u_n\}$ in $\oR^r$. Moreover,  let $U$ be the $r\times n$ matrix with columns $u_1,\ldots,u_n$.
Set 
\begin{equation}\label{relU}
U_{ij}={u_iu_j^\sfT+u_ju_i^\sfT\over 2},\ \UU_V=\lan U_{ii}: i\in V\ra,\ \UU_E=\lan U_{ij}: \{i,j\}\in E\ra \subseteq \SSS^r.
\end{equation}
If $x$ is an extreme point of $\EE(G)$, then $\UU_E\subseteq  \UU_V$. 
\end{lemma}

\begin{proof}
Assume that  $\UU_E\not\subseteq \UU_V$.
Then there exists  a matrix $R\in  \UU_V^\perp\setminus   \UU_E^\perp$.
As $R\in \UU_V^\perp$, the matrix $Z= U^\sfT RU=(\lan R,U_{ij}\ra)_{i,j=1}^n \in \SSS^n$ is a perturbation of $X$ (recall (\ref{relpert}) and (\ref{relU})).
As $R\not\in \UU_E^\perp$, $Z_{ij}\ne 0$ for some edge $\{i,j\}\in E$.
Now, $X\pm \epsilon Z\in \EE_n$ for some $\epsilon>0$. Hence,
$x$ can be written as the convex combination 
$(\pi_E(X+\epsilon Z) + \pi_E(X-\epsilon Z))/2$, where $\pi_E(X\pm \epsilon Z)$ are distinct points of $\EE(G)$. This contradicts the assumption that $x$ is an extreme point of $\EE(G)$.
\end{proof}

Given $x\in \EE(G)$, its {\em fiber}  is the set of all psd completions of $x$ in $\EE_n$, i.e., 
$$\fib(x)=\{X\in \EE_n: \pi_E(X)=x\}.$$

We close this section with a  simple but  useful  lemma about extreme points of projected elliptopes.

\begin{lemma}\label{lemextG2}
For  a  vector  $x \in \EE(G)$ we have that    
\begin{itemize}
\item[(i)] $x\in \ext \EE(G)\ $ if and only if   $\ \fib(x)$ is a face of $\EE_n$.
\item[(ii)] If $x \in \ext \EE(G)$  then $\ext \fib(x) \subseteq \ext \EE_n$. 
\end{itemize}
\end{lemma}

\begin{proof}$(i)$ Say $x\in \ext \EE(G)$ and let $\lambda A+(1-\lambda)B \in \fib(x)$, where $A,B \in \EE_n$ and $\lambda \in (0,1)$. Then $x=\lambda \pi_E(A)+(1-\lambda)\pi_E(B) \in \EE(G)$ and since $x \in \ext \EE(G)$ this implies that  $A,B \in \fib(x)$. The other direction is similar.

$(ii)$ The assumption combined with $(i)$ imply that $\fib(x)$ is a face of $\EE_n$ 
 and using \eqref{extremepoints} the claim follows.
\end{proof}

\section{Properties of the extreme Gram dimension} \label{secprop}

\subsection{Minor-monotonicity and clique sums}\label{secminor}

In this section we  investigate the behavior of the graph parameter $\egd(\cdot)$ under the  graph operations of  taking minors and clique sums.

\begin{lemma}\label{lemminor}
The parameter $\egd(\cdot)$ is minor monotone. That is,  for any edge $e$ of~$G$, 
$$\egd(G\backslash e)\le \egd(G)\  \text{ and } \ \egd (G\slash e)\le \egd (G).$$ 
\end{lemma}

\begin{proof}Consider a graph  $G=([n],E)$ and an edge $e\in E$, and set  $r=\egd(G)$.   We begin by showing  that   $\egd(G\setminus e)\le r$. Using Lemma \ref{def:egdgeom},  it suffices to show   that $\EE(G\setminus e)\subseteq \pi_{E\setminus e}(\conv(\EE_{n,r}))$. For this, let $x \in \EE(G\setminus e)$ and choose a scalar  $x_e~\in [-1,1]$ such that $(x,x_e)\in \EE(G)$.  Since $\egd(G)=r$ it follows that $(x,x_e)~\in \pi_{E}(\conv(\EE_{n,r}))$ and thus $x \in \pi_{E\setminus e}(\conv(\EE_{n,r}))$.

We now show that $\egd(G/e)\le r$. Say,  $e$ is the edge $(n-1,n)$ and set $G/e=([n-1],E')$. By Lemma  \ref{def:egdgeom}, it suffices to show  $\EE(G/e) \subseteq \pi_{E'}(\conv(\EE_{n-1,r}))$.  For this, let  $x \in \EE(G/e)$; we show that $x\in \pi_{E'}(\conv(\EE_{n-1,r}))$.
As $x$ belongs to $\EE(G/e)$, it follows that   $x=\pi_{E^{'}}(X)$ for some matrix  $X \in \EE_{n-1}$.
Say,    $X={\rm Gram}(p_1,\ldots,p_{n-1})$ for some vectors $p_1,\ldots,p_{n-1}$.
Let $X[\cdot,n-1]$ denote  the last column of $X$ and define the new matrix 
  $$Y=\left(\begin{matrix}
X & X[\cdot,n-1]\\
X[\cdot,n-1]^\sfT & 1
\end{matrix}\right)\in \SSS^n.$$ 
Then,
  $Y_{n-1,n}=X_{n-1,n-1}=1$ holds.
Moreover,  $Y={\rm Gram}(p_1,\ldots,p_{n-1},p_{n-1})$, which shows that $Y\in \EE_n$. Therefore, the projected vector  $y=\pi_{E}(Y)$ belongs to $ \EE(G)$ and its $(n-1,n)$-coordinate  satisfies: $y_{n-1,n}=Y_{n-1,n}=1$.
As $y\in~\EE(G)$ with   $\egd(G)=r$, it follows from Lemma \ref{def:egdgeom} that  there exist matrices $Y_1,\ldots,Y_m \in \EE_{n,r}$ and scalars $\lambda_i > 0$ with $\sum_{i=1}^m\lambda_i=1$ satisfying  
 \begin{equation}\label{werchrist}
 y=\pi_E(\sum_{i=1}^m \lambda_iY_i).
 \end{equation}
 Combining with $y_{n-1,n}=1$, this implies:
 \begin{equation}\label{gamiseme}
 1=\sum_{i=1}^m \lambda_i (Y_i)_{n-1,n}.
 \end{equation}
  Since the matrices $Y_i \ (i \in [m])$ are psd with diagonal entries equal to 1, all their  entries are bounded in absolute value by 1. Moreover, as $\lambda_i \in (0,1]$ for every $i\in [m]$,  \eqref{gamiseme} implies that $(Y_i)_{n-1,n}=1$ and thus 
  $$Y_i[\{n-1,n\}]=\left(\begin{matrix} 1 & 1 \cr 1 & 1 \end{matrix}\right).$$
  Therefore, the vector $(1,-1)$ lies in the kernel of $Y_i[\{n-1,n\}]$.
  Using (\ref{relker}), we can conclude that  the last two columns of $Y_i$ indexed by $n-1$ and by $n$ are equal.
 
 For $i\in [m]$, let   $X_i$ be  the matrix   obtained from $Y_i$ by removing its $n$-th row and its $n$-th column.  Since $X_i$ is a submatrix of $Y_i$ we have that  $\rankspace X_i \le \rankspace Y_i\le r.$ Set $X=\sum_{i=1}^m \lambda_iX_i$ and notice that, by construction, it belongs to $\conv (\EE_{n-1,r})$.   Moreover, since ${Y_i}[\cdot ,n-1]={Y_i}[\cdot,n]$ for all $ i \in [m]$,  it follows that $x=\pi_{E^{'}}(X)$. Lastly, since $X \in \conv (\EE_{n-1,r})$,  it follows  that $x \in   \pi_{E'}(\conv (\EE_{n-1,r}))$. This concludes the proof that    $\egd(G\slash e)\le r$.
\end{proof}

We now recall a well known  useful fact concerning completions of psd matrices. We include a short proof for completeness. 
\begin{lemma}\label{lem:cliquepsd} Consider two psd matrices $X_i$ indexed respectively by $V_i$ for $i=1,2$. Assume that  $X_1[V_1\cap V_2]=X_2[V_1\cap V_2]$. Then $X_1$ and $X_2$  admit a common psd completion $X$ indexed by $V_1\cup V_2$ with  $\rank X=
\max\{\rankspace (X_1),\rankspace (X_2)\}$. 
\end{lemma}

\begin{proof}
Set $r=\max\{\rankspace (X_1),\rankspace (X_2)\}$. Let $u^{(i)}_j$ ($j\in V_i$) be a Gram representation of $X_i$ (for $i=1,2$) and assume   without loss of generality that the two families of vectors  lie in the same space $\oR^r$. Then, there exists an orthogonal $r\times r$ matrix $Q$  mapping $u^{(1)}_j$ to $u^{(2)}_j$ for $j\in V_1\cap V_2$. Clearly, the  Gram matrix of the vectors  of $Qu^{(1)}_j$ ($j\in V_1$) together with $u^{(2)}_j$ ($j\in V_2\setminus V_1$) is a common psd completion with rank at most $r$.
\end{proof}

As a direct application of the above lemma, we obtain the followin result of \cite{LV12}: if $G$ is the clique sum of $G_1$ and $G_2$,  then its  Gram dimension satisfies: 
  $\gd(G)=\max\{\gd(G_1),\gd(G_2)\}$.
For the extreme Gram dimension, the analogous  result holds only for clique $k$-sums with $k\le 1$.

\begin{lemma}\label{lemcliquesum}
Let $G_1=(V_1,E_1)$ and $G_2=(V_2,E_2)$ be graphs.
If $|V_1\cap V_2|\le 1$ then the clique sum $G$ of $G_1,G_2$ satisfies $\egd(G)=\max\{\egd(G_1),\egd(G_2)\}.$
\end{lemma}

\begin{proof}Let $x\in \EE(G)$ and set $r=\max\{\egd(G_1),\egd(G_2)\}$. We will  show that $x\in \pi_{E}(\conv (\EE_{n,r}))$. 
For $i=1,2$, the vector  $x_i=\pi_{E_i}(x)$ belongs to $\pi_{E_i}(\conv(\EE_{|V_i|,r}))$. Hence,
$x_i=\pi_{E_i}(\sum_{j=1}^{m_i} \lambda_{i,j} X^{i,j})$ for some  $X^{i,j}\in \EE_{|V_i|,r}$ and $\lambda_{i,j}\ge 0$ with $\sum_j \lambda_{i,j}=1$.
As $|V_1\cap V_2|\le 1$, any two matrices $X^{1,j}$ and $X^{2,k}$ share at most one diagonal entry, equal to 1 in both matrices.
 By Lemma~\ref{lem:cliquepsd}, $X^{1,j}$ and $X^{2,k}$ have a common completion $Y^{j,k}\in \EE_{n,r}$.
This implies that $x=\pi_E(\sum_{j=1}^{m_1}\sum_{k=1}^{m_2} \lambda_{1,j}\lambda_{2,k} Y^{jk})$,  which shows 
$x\in\pi_{E}(\conv (\EE_{n,r}))$.
\end{proof}

Throughout this paper  we denote by $\GG_r$ the class of graphs having extreme Gram dimension at most $r$. By  Lemma \ref{lemminor} and Lemma \ref{lemcliquesum}  the class $\GG_r$ is closed under taking disjoint unions  and clique 1-sums of graphs. Nevertheless, it is {\em not} closed under clique  $k$-sums when $k\ge 2$.  For example,  the graph $F_3$ seen in   Figure~\ref{fig:frminors}   is a clique 2-sum of triangles, however $\egd(F_3)=3$ (Theorem \ref{theoFr}) while  triangles  have extreme Gram dimension 2  (Lemma~\ref{lemrankr}).  


\subsection{Upper and lower bounds}\label{secbound}

From Lemma \ref{eq:reforegd}  we know that, for any graph $G$,
\begin{equation*}
 \egd(G)=\max_{x \in \ext\EE(G)}\gd(G,x).
 \end{equation*}
According to this characterization,  in order to show that $\egd(G)\le r$,  it suffices to show  that every partial matrix  $x \in \ext \EE(G)$ has a psd completion of rank at most $r$. Using  results from Section \ref{sec:geom}  we obtain the following:



\begin{lemma}\label{lemrankr}
The extreme Gram dimension of the complete graph $K_n$ is
\begin{equation}\label{eq:boundcomplete}
\egd(K_n) =\max\left\{r\in\mathbb{Z}_+: {r+1\choose 2}\le n\right\}= \left\lfloor {\sqrt{8n+1}-1\over 2}\right\rfloor.
\end{equation}
Hence,  for any graph $G$ on $n$ nodes we have that:
\begin{equation}\label{eq:upperbound}
\egd(G)\le \max\left\{r\in\mathbb{Z}_+: {r+1\choose 2}\le n\right\}= \left\lfloor {\sqrt{8n+1}-1\over 2}\right\rfloor.
\end{equation}
\end{lemma}

\begin{proof}Notice that  $\EE(K_n)$ is the bijective image of $\EE_n$ in $\oR^{\binom{n}{2}}$, obtained by considering only the upper triangular part of matrices in $\EE_n$. Then, for any $X \in~\ext \EE_n$ with $\rankspace X=r$ we know that  $\binom{r+1}{2} \le n$ (recall \eqref{eqdimextreme}). Moreover, from Theorem \ref{dcfv} we know that for any natural number $r$ satisfying $\binom{r+1}{2}\le n$ there exists an extreme point of $\EE_n$ with rank equal to $r$.  The second part follows from \eqref{eq:boundcomplete} using  the fact that $\egd(\cdot)$ is minor monotone (Lemma~\ref{lemminor}).
\end{proof}

The following lemma is a direct consequence  of \eqref{eq:upperbound}. 

\begin{lemma}\label{lem:uboundegd} Consider a graph $G$ with  $|V(G)|=\binom{r+1}{2}$. Then  $\egd(G) \le r$. 
\end{lemma}

On the other hand, in order to obtain a lower bound  $\egd(G)\ge r$, 
we need  an extreme point of $\EE(G)$,  all of whose positive semidefinite completions have rank at least $r$. 
This poses two difficulties: how to construct a suitable  extreme point of $\EE(G)$ and then,  given an extreme point of   $ \EE(G)$, how to verify that {\em all} its positive semidefinite completions have rank at least $r$. 
We resolve this by using the construction of extreme points from Lemma \ref{lemextG2}. Indeed, if we can find a point $x\in \EE(G)$ admitting a {\em unique} completion $X\in \EE_n$ which is an extreme point of $\EE_n$ and has rank $r$, then we can conclude that $x$ is an extreme point of $ \EE(G)$ which has Gram dimension  $r$, thus showing that  $\egd(G)\ge r$.
We summarize this in the following lemma  for further reference.

\begin{lemma}\label{lem:basic} 
Assume  that there exists  $x\in \EE(G)$ which has a unique completion $X\in \EE_n$. Assume moreover that  $X$ is an extreme point of $\EE_n$ and that $X$ has rank $r$. Then, $\egd(G)\ge r$.
\end{lemma}

\begin{proof}As $\fib(x)=\{X\}$ and $X \in \ext \EE_n$, it follows that  $\fib(x)$  is a face of $\EE_n$ and then Lemma~\ref{lemextG2} implies that $x \in \ext \EE(G)$.
\end{proof}

\subsection{The strong largeur d'arborescence}\label{secupperbound}



In this section we introduce a new treewidth-like  parameter that will serve as   an upper bound for the extreme Gram dimension. 

\begin{definition}\label{def:strongla} The {\em strong largeur d'arborescence} of \index{strong largeur d'arborescence} a graph $G$, denoted by $\sla(G)$, is the smallest integer $k\ge 1$ for which $G$ is a minor of $T\stp K_k$  for some tree $T$.
\end{definition}
Notice the analogy with the largeur d'arborescence where  the Cartesian product has been substituted  with the strong graph product. 
It is clear from its definition that the parameter $\sla(\cdot)$ is minor monotone. Moreover, \index{strong largeur d'arborescence}
 
\begin{lemma}\label{lem:bound}For any graph $G$ we have that 

$$ \frac{\tw(G)+1}{2} \le \sla(G)\le \lda(G).$$
\end{lemma}

\begin{proof}The rightmost inequality follows directly from the definitions. For the leftmost inequality assume that $\sla(G)=k$, i.e., $G$ is minor of $T\stp K_k$ for some tree $T$. Notice that the graph $T\stp K_k$ can be obtained by taking clique  $k$-sums of copies of the graph $K_2\stp K_k$. 
By \eqref{tr:csum}, $\tw(T \stp K_k) =\tw(K_2\stp K_k)=2k-1$. Combining this with the fact that  the treewidth  is  minor-monotone, we obtain that $\tw(G)\le \tw(T\stp K_k)=2k-1$ and   the claim follows. 
\end{proof}

Our main goal in this  section is to  show that the extreme Gram dimension is upper bounded by the strong largeur d'arborescence: $\egd(G)\le \sla(G)$ for any graph $G$.  As we will see in later sections, this property 
will play  a crucial role in characterizing graphs with extreme Gram dimension at most 2.
We start with a technical  lemma which we need for the proof of Theorem \ref{theomainr} below.

\begin{lemma}\label{lemr}
Let   $\{u_1,\ldots, u_{2r}\}$ be a set of vectors, denote its rank by $\rho$.
Let $\UU$ denote the linear span of the matrices  $U_{ij}=(u_iu_j^\sfT+u_ju_i^\sfT)/2$ for all $i,j\in \{1,\ldots,r\}$ and all $i,j\in \{r+1,\ldots, 2r\}$.
If $\rho\ge r+1$ then $\dim \UU <{\rho+1\choose 2}$.
\end{lemma}

\begin{proof}Let $I\subseteq \{1,\ldots, r\}$ for which $\{u_i: i\in I\}$ is a maximum linearly independent subset of $\{u_1,\ldots, u_r\}$ and let $J\subseteq \{r+1,\ldots,2r\}$ such that the set $\{u_i: i\in I\cup J\}$ is maximum  linearly independent; thus  $|I|+|J|=\rho$.
Set $K=\{1,\ldots,r\}\setminus I$, $ L=\{r+1,\ldots,2r\}\setminus J$, and $J'=J\setminus \{k\}$, where  $k$ is some given (fixed) element of $J$.
For  any $l\in L$, there exists scalars $a_{l,i}\in\oR$ such that 
\begin{equation}\label{eql}
u_l=\sum_{i\in I\cup J'} a_{l,i} u_i + a_{l,k}u_k.
\end{equation}
Set 
$$A_l= \sum_{i\in I\cup J'} a_{l,i} U_{ik}\ \ \text{ for } \  l\in L.$$
Then, define the set  $\WW$  consisting of the matrices $U_{ii}$ for $i\in I\cup J$,
$U_{ij}$ for all $i\ne j $ in $I\cup J'$,
$U_{kj}$ for all $j\in J'$, and $A_l$ for all $l\in L$.
Then, $ |\WW|= \rho + {\rho-1\choose 2} + r-1= {\rho\choose 2} +r = {\rho+1\choose 2}+ r-\rho \le {\rho+1\choose 2}-1.$
In order to conclude the proof it suffices to show that $\WW$ spans the space $\UU.$

Clearly, $\WW$ spans all matrices $U_{ij}$ with $i,j\in \{1,\ldots, r\}$. Moreover, by its definition  $\WW $ contains all matrices $U_{ij}$ for $i,j\in J$. Consequently, it remains to show that  $U_{kl} \in \WW$ for all $l \in L$, $U_{lj}\in \WW$ for all $l \in L$ and $j\in J'$  and that  $U_{ll'}\in \WW$ for all $l,l' \in L$.
Fix $l\in L$. Using (\ref{eql}) we obtain that $U_{lk} =A_l+a_{l,k}U_{kk}$ lies in the span of $\WW$.
Moreover, for $j\in J'$, $U_{lj}=\sum_{i\in I\cup J'}a_{l,i}U_{ij}+a_{l,k}U_{kj}$ also lies in the span of $\WW$.
Finally, for $l'\in L$,
$U_{ll'}= \sum_{i,j\in I\cup J'}a_{l,i}a_{l',j}U_{ij} + a_{l',k}A_l+a_{l,k}A_{l'} +a_{l,k}a_{l',k}U_{kk}$ is also spanned by $\WW$.
This concludes the proof.
\end{proof}

\begin{lemma}\label{sdrtghui}Let $v_1,\ldots,v_n$ be a family of linearly independent vectors in $\oR^n$. Then the matrices  $(v_iv_j^\sfT+v_jv_i^\sfT)$ for  $1\le i\le j \le n$ span $\mcs^n$.
\end{lemma}

\begin{proof}Consider a matrix $Z \in \mcs^n$ such that $\la Z,(v_iv_j^\sfT+v_jv_i^\sfT)/2\ra =0$ for all $i,j \in [n]$; we show that $Z$ is the zero matrix. For any vector $x \in \oR^n$, we can write $x=\sum_{i=1}^n\lambda_iv_i$ for some scalars $\lambda_i \ (i \in [n])$ and thus $x^\sfT Zx=0$. This implies  $Z=0$.
\end{proof}

We can now show   the main result of this section.

\begin{theorem}\label{theomainr}
For any tree $T$,  we have that $\egd(T\stp K_r)\le r$. 
\end{theorem}

\begin{proof}Let $G=T\stp K_r$, where $T$ is a tree on $[t]$. Say, $G=(V,E)$ with $|V|=n$. So the node set of $G$ is $V=\cup_{i=1}^t V_i$, where the $V_i$'s are pairwise disjoint sets, each of cardinality $r$.
By definition of the strong product, for any edge $\{i,j\}$ of $T$, the set $V_i\cup V_j$ induces a clique in $G$, denoted as   $C_{ij}$. Then,  $G$ is the union of the cliques $C_{ij}$ over  all edges $\{i,j\}$ of $T$.
We show that $\egd(G)\le r$. For this,  pick an  element $x\in \ext\EE(G)$. Then  $x=\pi_E(X)$ for some $X\in \EE_n$.
As $C_{ij}$ is a clique in $G$,  the principal submatrix  $X^{ij}:=X[C_{ij}] $   is fully determined from $x$.
To show that $x$ has a psd completion of rank at most $r$, it suffices to show that $\rankspace X^{ij}\le r$ for all edges $\{i,j\}$ of $T$. Indeed,   by applying Lemma \ref{lem:cliquepsd},  we can then conclude the existence of a common psd completion of  the $X^{ij}$ of rank at most $r$. 

Pick an edge $\{i,j\}$ of $T$ and set $\rho=\rankspace X^{ij}$. Assume that $\rho\ge r+1$; we show  below that there exists a nonzero perturbation $Z$ of $X^{ij}$ such that 
\begin{equation}\label{eqZ}
\begin{array}{l}
Z_{hk}=0  \text{ for all }  (h,k)\in (V_i\times V_i) \cup (V_j\times V_j) ,\\   Z_{hk}\ne 0 \text{ for some } (h,k)\in V_i\times V_j.
\end{array}
\end{equation}
This permits to reach a contradiction: As $Z$ is a perturbation of $X^{ij}$, there exists $\epsilon>0$ for which $X^{ij}+\epsilon Z, \ X^{ij}-\epsilon Z\succeq 0$.
By construction,  $C_{ij}$ is the only maximal clique of $G$ containing the edges  $\{h,k\}$ of $G$ with $h\in V_i$ and $k\in V_j$.
 Hence, one can find  a psd completion $X'$ (resp., $X''$) of the matrix $X^{ij}+\epsilon Z$ (resp., $X^{ij}-\epsilon Z$) 
and the matrices  $X^{i'j'}$ for all edges $\{i',j'\}\ne \{i,j\}$ of $T$. Now, $x={1\over 2}(\pi_E(X')+\pi_E(X''))$, where 
$\pi_E(X'),\pi_E(X'')$ are distinct elements of $\EE(G)$, contradicting the fact that $x$ is an extreme point of $\EE(G)$.

We now construct the desired  perturbation $Z$ of $X^{ij}$ satisfying (\ref{eqZ}).  For this let $u_h$ ($h\in V_i\cup V_j$) be a Gram representation of $X^{ij}$ in $\oR^\rho$ and 
let $\UU\subseteq \SSS^\rho $ denote the linear span of the matrices 
$U_{hk}=(u_hu_k^\sfT+u_ku_h^\sfT)/2$ for all $h,k\in V_i$ and all $h,k\in V_j$.
Applying Lemma \ref{lemr}, as $\rho \ge r+1$, 
we deduce that $\dim \UU<{\rho+1\choose 2}$. 
Hence, by Lemma \ref{sdrtghui}, there exists a nonzero matrix $R\in\SSS^\rho$ lying in $ \UU^\perp$ for which the matrix 
 $Z\in \SSS^{2r}$ defined by 
$Z_{hk}=\la R, U_{hk}\ra$ for all $h,k\in V_i\cup V_j$, is nonzero.  By construction, $Z$ is a perturbation of $X^{ij}$ (recall Proposition \ref{propface}) and it satisfies $Z_{hk}=0$ whenever the pair $(h,k)$ is contained in $V_i$ or in $V_j$.  Moreover, 
as $Z\ne 0$, we have $Z_{hk}\ne 0$ for some $h\in V_i$ and $k\in V_j$.
Thus (\ref{eqZ}) holds and the proof is completed.
\end{proof}

\begin{corollary}\label{cor:egdlesla}
For any graph $G$, $\egd(G)\le \sla(G)$.
\end{corollary}

\begin{proof}If $\sla(G)=k$, then  $G$ is a minor of $T\stp K_k$ for some tree $T$ and thus   $\egd(G)\le \egd(T\stp K_k)\le k$, by Lemma \ref{lemminor} and Theorem \ref{theomainr}.
\end{proof}

\section{The extreme Gram dimension of some graph classes}\label{sec_forbminor}

In this section we construct three classes of graphs $F_r$, $G_r$, $H_r$, 
whose extreme Gram dimension is equal to $r$. Therefore, they are forbidden minors for the class $\GG_{r-1}$
 of graphs with extreme Gram dimension at most $r-1$. 
 As we will see in the next section, this gives all the forbidden minors for   the class $\GG_2$.

 The graphs $G_r$ were already considered by Colin de Verdi\`ere \cite{V98} in relation to the graph parameter $\nu(\cdot)$, to which we will come back in Section~\ref{seclast}. Each of the graphs $G=F_r$, $G_r$, $H_r$  has ${r+1\choose 2}$ nodes and thus their extreme Gram dimension is at most $r$ (recall Lemma \ref{lem:uboundegd}). Moreover, they  satisfy:    $\egd(G\slash e)\le r-1$  after contracting any edge $e$.
In order to  show equality $\egd(G)=r$,
we will rely on Lemma \ref{lem:basic}.


To use  Lemma~\ref{lem:basic}  we need tools permitting to show existence of a {\em unique} completion for a partial matrix $x\in \EE(G)$. 
We introduce below  such a tool: `forcing a non-edge with a  singular clique'. This is based on  the following property, which is a special case of relation (\ref{relker}):

\begin{equation}\label{lem:psd} 
\left(\begin{matrix}
A & b\cr
b^\sfT & \alpha\
\end{matrix}\right) \succeq 0 
\Longrightarrow  b^\sfT u=0 \ \ \forall u\in \Ker A.
\end{equation}

%
%
\begin{lemma}\label{lemforce}
 Let $x\in \EE(G)$, let $C\subseteq V$ be a clique of $G$ and  let $\{i,j\}\not\in E(G)$ with $i\not\in C$, $j\in C$.
Set $x[C]=(x_{ij})_{i,j\in C}\in \EE_{|C|}$ (setting $x_{ii}=1$ for all $i$).   Assume  that $i$ is adjacent to all nodes of $C\setminus \{j\}$,  $x[C]$ is singular and $x[C\setminus \{j\}]$  is nonsingular.
Then, for any psd completion $X$ of $x$,  the $(i,j)$-th entry $X_{ij}$  is uniquely determined. 
\end{lemma}

\begin{proof}Let $X$ be a psd completion of $x$. 
The principal submatrix $X[C\cup \{i\}]$ has the block form shown in (\ref{lem:psd}) (with $A$ being indexed by $C$), where all entries are specified (from $x$) except the  entry $b_{j}=X_{ij}$ which is unspecified since $\{i,j\}\not\in E(G)$. As $x[C]$ is singular there exists a nonzero vector $u$ in the kernel of $x[C]$. Moreover,   since $x[C\setminus\{j\}]$ is nonsingular if follows that  $u_j\ne 0$. Hence the condition $b^\sfT u=0$  permits to derive 
the value of $X_{ij}$ from $x$.
\end{proof}

When applying Lemma \ref{lemforce} we will say   that ``the clique $C$ forces the pair $\{i,j\}$''.
The  lemma will be used  in an iterative manner: Once a non-edge $\{i,j\}$ has been forced, we know the value $X_{ij}$ in any psd completion $X$ and thus we can replace $G$ by $G+\{i,j\}$ and search for a new forced pair in the extended graph $G+\{i,j\}$. 

We note in passing that a general  framework for constructing partial psd matrices with a unique psd completion  has been developed in \cite{LV13}.  Following that approach,  most of the constructions described  below can be easily recovered.

\subsection{The class $F_r$}\label{secFr}
For $r\ge 2$ the graph $F_r$ has $r+{r\choose 2}={r+1\choose 2}$ nodes, denoted as $v_i$ (for $i\in [r]$) and $v_{ij}$ (for $1\le i<j\le r$); it consists of a clique $K_r$ on the nodes $\{v_1,\ldots,v_r\}$ together with the cliques $C_{ij}$  on $\{v_i,v_j,v_{ij}\}$ for all $1\le i<j\le r$.  
The  graphs  $F_3$ and $F_4$ are  illustrated in Figure~\ref{fig:frminors}.

\begin{figure}[!h]
\bc  \includegraphics[scale=0.4]{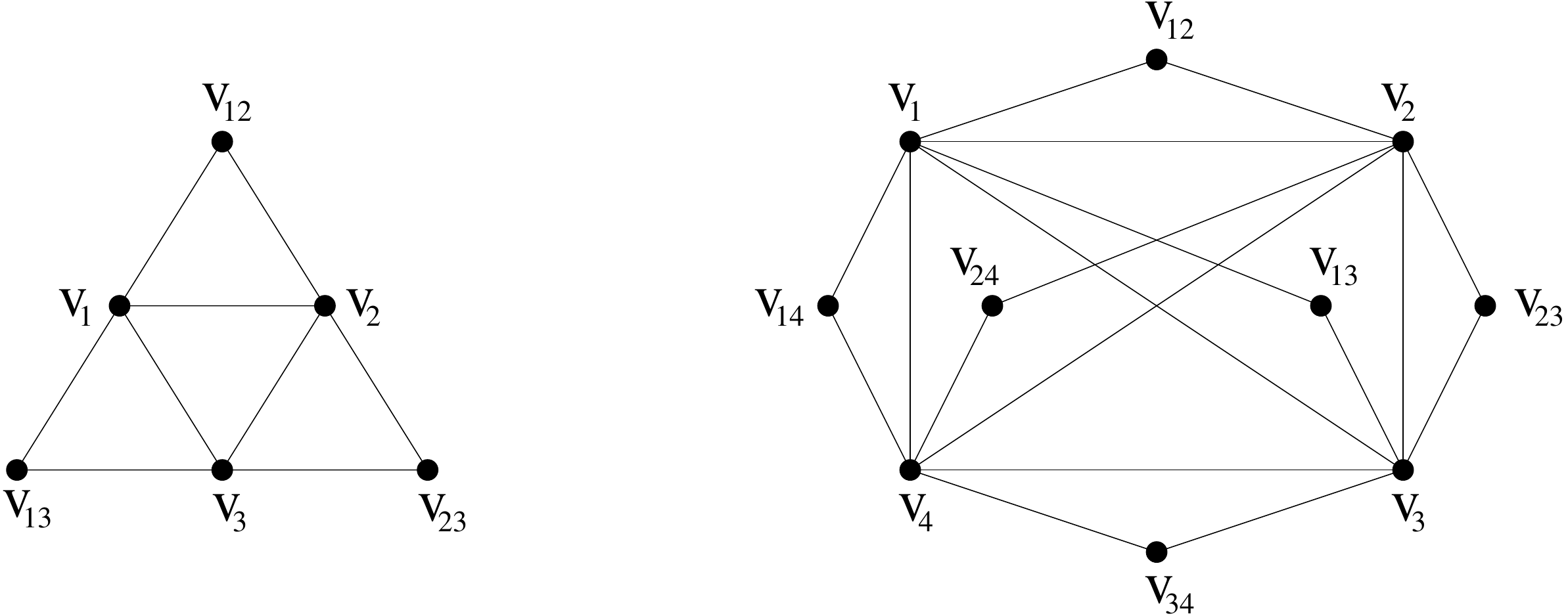}
 \caption{The graphs $F_3$ and $F_4$.}\label{fig:frminors}
\ec
 \end{figure}
For $r=2$, $F_2=K_3$ has extreme Gram dimension 2.  More generally:

\begin{theorem}\label{theoFr}
For $r\ge 2$, $\egd(F_r)=r$. Moreover, $F_r$ is a minimal forbidden minor for the class $\GG_{r-1}$. 
\end{theorem}

\begin{proof}Since $F_r$ has $\binom{r+1}{2}$ nodes it follows from Lemma \ref{lem:uboundegd} that $\egd(F_r)\le r$.  We now  show that $\egd(F_r)\ge r$. 
For this we label the nodes $v_1,\ldots,v_r$ by the standard unit vectors $e_1,\ldots,e_r\in \oR^r$ and  $v_{ij}$ 
by the vector 
$(e_i+e_j)/\sqrt 2$. 
Consider the Gram matrix $X$ of these $n={r+1\choose 2}$ vectors  and its projection $x=\pi_{E(F_r)}(X)\in \EE(F_r)$. 
Using \eqref{eqdimextreme0} it follows directly that  $X$ is an extreme point of $\EE_n$. 
We now show that $X$ is the  only psd completion of $x$ which, in view of Lemma \ref{lem:basic}, implies that
$\egd(F_r)\ge r$.
For this we use Lemma \ref{lemforce}. Observe that, for each $1\le i<j\le r$, the matrix 
$x[C_{ij}]$ is singular.
First, for any $k\in [r]\setminus \{i,j\}$, the clique $C_{ij}$ forces the non-edge $\{v_k,v_{ij}\}$
and then, for any other $1\le i'<j'\le r$, the clique $C_{ij}$ forces the non-edge $\{v_{ij},v_{i'j'}\}$.
Hence, in any psd completion of $x$, all the entries indexed by non-edges are uniquely determined, i.e.,
$\fib(x)=\{X\}$. 

Next, we show minimality. Let $e$ be an edge of $F_r$, 
we  show  that $\egd(H)\le r-1$ where $H=F_r\backslash e$.
If $e$ is an edge of the form $\{v_i,v_{ij}\}$, then $H$ is the clique 1-sum of an edge and a graph on ${r+1\choose 2}-1$ nodes and thus $\egd(H)\le r-1$ follows using Lemmas \ref{lemcliquesum} and \ref{lemrankr}. Suppose now  that $e$ is contained in the central clique $K_r$, say
 $e=\{v_1,v_2\}$. 
We show that $H$ is contained in a graph of the form $T\stp K_{r-1}$ for some tree $T$. We choose $T$ to be the star $K_{1,r-1}$ and we 
  give a suitable partition of the nodes of $F_r$ into sets $V_0\cup V_1\cup\ldots \cup V_{r-1}$, where each $V_i$ has cardinality at most $r-1$, $V_0$ is  assigned to the center node of the star  $K_{1,r-1}$ and $V_1,\ldots,V_{r-1}$ are  assigned to the $r-1$ leaves of $K_{1,r-1}$. 
Namely, set $V_0=\{v_{12}, v_3,\ldots, v_r\}$,
$V_1=\{v_1,v_{13},\ldots, v_{1r}\}$,
$V_2=\{v_2,v_{23},\ldots, v_{2r}\}$ and, for $k\in \{3,\ldots, r-1\}$,
$V_{k}=\{v_{kj}: k+1\le j\le r\}$.
Then, in the graph $H$,  each edge is contained in one of the sets $V_0\cup V_k$ for $1\le k\le r-1$. This shows that $H$ is a subgraph of $K_{1,r-1}\stp K_{r-1}$ and thus $\egd(H)\le r-1$ (by Theorem~\ref{theomainr}).
\end{proof}

As an application of Theorem \ref{theoFr} we get:

\begin{corollary}
If the tree $T$ has a node of degree at least $(r-1)/2$ then $$\egd(T\stp K_r)=r.$$
\end{corollary}

\begin{proof}
Directly from Theorem \ref{theoFr}, as $T\stp K_r$ contains a subgraph $F_r$.
\end{proof}

\subsection{The class $G_r$}\label{secGr}
Consider an equilateral triangle and subdivide each side into $r-1$ equal segments. Through these points draw line segments  parallel to the sides of the triangle. This construction creates a triangulation of the big triangle into $(r-1)^2$ congruent equilateral triangles.  The graph $G_r$ corresponds to the edge graph of this triangulation. The graph $G_5$ is illustrated in Figure~\ref{fig:t5}. 

 \begin{figure}[h]
\bc  \includegraphics[scale=0.5]{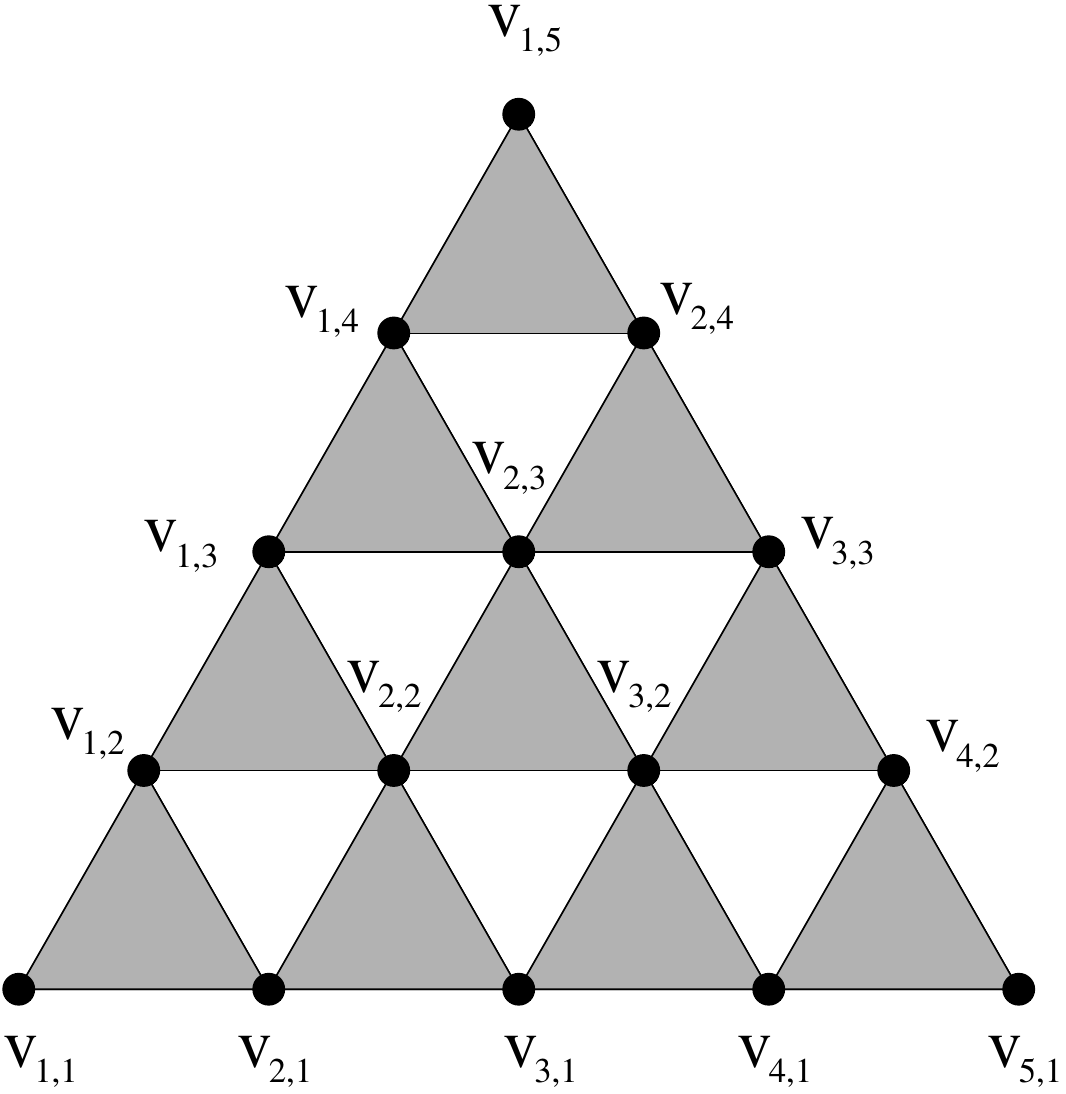}
 \caption{The graph $G_5$.}\label{fig:t5}
\ec
 \end{figure}

The graph $G_r$  has $\binom{r+1}{2}$ vertices,  denoted as  $v_{i,l}$ for $l\in [r]$ and $i\in [r-l+1]$ (with $v_{1,l},\ldots,v_{r-l+1,l}$ at level $l$, see Figure \ref{fig:t5}).
Note that $G_2=K_3=F_2$, $G_3=F_3$, but $G_r\ne F_r$ for  $r\ge 4$.
Using  the following lemma we can construct some points of $\EE(G_r)$ with a unique completion.
 
 \begin{lemma}\label{claim0}
Consider a labeling of the nodes of $G_r$ by vectors $w_{i,l}$ satisfying the following property $(P_r)$:
For each triangle 
$C_{i,l}=\{v_{i,l}, v_{i+1,l}, v_{i,l+1}\}$ of $G_r$, 
the set 
$\{w_{i,l},w_{i+1,l},w_{i,l+1}\}$   is minimally linearly dependent. (These triangles are shaded in Figure \ref{fig:t5}).
Let $X$ be the Gram matrix of the vectors $w_{i,l}$ and let   $x=\pi_{E(G_r)}(X)$ be its projection. Then $X$ is  the unique completion of $x$.
\end{lemma}

\begin{proof}For $r=2$, $G_2=K_3$ and there is nothing to prove.
Let $r\ge 3$ and assume that the claim holds for $r-1$. Consider a labeling $w_{i,l}$ of  $G_r$ satisfying  $(P_r)$  and the corresponding vector $x\in \EE(G_r)$. We show, using Lemma \ref{lemforce},  that the entries $Y_{uv}$ of a psd completion $Y$ of $x$ are uniquely determined for all 
$\{u,v\}\not\in E(G_r)$. 
For this, denote by $H, R, L$ the sets of nodes lying on 
the `horizontal' side, the `right' side and the `left' side of $G_r$, respectively (refer to the drawing of $G_r$ of Figure \ref{fig:t5}).
Observe that each of $G_r\backslash H$, $G_r\backslash R$, $G_r\backslash L$ is a copy of $G_{r-1}$. As the induced vector labelings on each of these graphs satisfies the property $(P_{r-1})$, we can conclude  using the induction assumption  that the entry $Y_{uv}$ is uniquely determined whenever the pair $\{u,v\}$ is contained in the vertex set of one of  $G_r\backslash H$, $G_r\backslash R$, or $G_r\backslash L$.
The only non-edges $\{u,v\}$   that are not yet covered arise when $u$ is a corner of $G_r$ and $v$ lies on the opposite side, say $u=v_{1,1}$ and $v=v_{r-l+1,l}\in R$.
If $l\ne 1,r$ then the clique $C_{1,1}=\{v_{1,1},v_{2,1},v_{1,2}\}$ forces the pair $\{u,v\}$ (since 
$\{v,v_{1,2}\}\subseteq E(G_r\backslash H)$ and $\{v,v_{2,1}\}\subseteq E(G_r\backslash L)$). 
If $l=r$ then the clique $C_{1,r-1}=\{v_{1,r-1},v_{2,r-1},v_{1,r}\}$ forces the pair $\{u,v\}$ (since 
$\{u,v_{1,r-1}\}\subseteq E(G_r\backslash R)$ and the value at the pair $\{u,v_{2,r-1}\}$ has just been specified).
Analogously for the case $l=1$. 
This concludes the proof. 
 \end{proof}

 \begin{theorem}
We have that  $\egd(G_r)=r$ for all $r\ge 2$. Moreover,  $G_r$ is a minimal forbidden minor for the class $\GG_{r-1}$.
 \end{theorem}
 
 \begin{proof}Since $G_r$ has $\binom{r+1}{2}$ nodes it follows from Lemma \ref{lem:uboundegd} that $\egd(G_r)\le r$. We now  show  that $\egd(G_r)\ge r$. 
For this, choose a vector labeling of the nodes of $G_r$ satisfying the conditions of Lemma \ref{claim0}: 
Label the nodes $v_{1,1},\ldots,v_{r,1}$ at level $l=1$ by the standard unit vectors $w_{1,1}=e_1,\ldots,w_{r,1}=e_r$ in $\oR^r$ and 
define inductively  $w_{i,l+1}= {w_{i,l}+w_{i+1,l}\over \|w_{i,l}+w_{i+1,l}\|}$ for $l=1,\ldots,r-1$.
By Lemma \ref{claim0} their Gram matrix $X$ is the unique completion of its projection $x=\pi_{E(G_r)}(X)\in \EE(G_r)$. 
Moreover, $X$ is extreme in $\EE_n$ since $\UU_V$ is full-dimensional in $\SSS^r$. This shows $\egd(G_r)\ge r$, by Lemma \ref{lem:basic}.

We now show  
 that  $\egd(G_r\backslash e)\le r-1$. For this use the following inequalities: 
 $\egd(G_r\backslash e)\le \sla(G_r\backslash e)\le  \lda(G_r\backslash e) \le r-1$, where  the leftmost inequality follows from Corollary~\ref{cor:egdlesla} and the rightmost  one is shown in \cite{Ko00}. 
    \end{proof}

 

We conclude with two immediate corollaries.

   \begin{corollary} The graph parameter $\egd(G)$ is unbounded for the class of planar graphs.
\end{corollary}  
 
\begin{corollary}
Let $T$ be a tree which contains a path with $2r-2$ nodes. 
Then, $\egd(T\stp K_r)=r$.
\end{corollary}

\begin{proof}It is shown in \cite{V98} that $G_r$ is a minor of the Cartesian product of  two paths $P_r$ and  $P_{2r-2}$ (with, respectively, $r$ and $2r-2$ nodes). Hence,
$G_r \preceq P_{2r-2} \Box P_{r}\preceq T\stp K_r$ and thus $r=\egd(G_r)\le \egd(T\stp K_r)$.
\end{proof}

\subsection{The class $H_r$}\label{secHr}
In this section we consider a third class of graphs  $H_r$ for every  $r\ge 3$. In order to explain the   general definition we first describe the base case $r=3$. 

\begin{figure}[h!]
\bc  \includegraphics[scale=0.45]{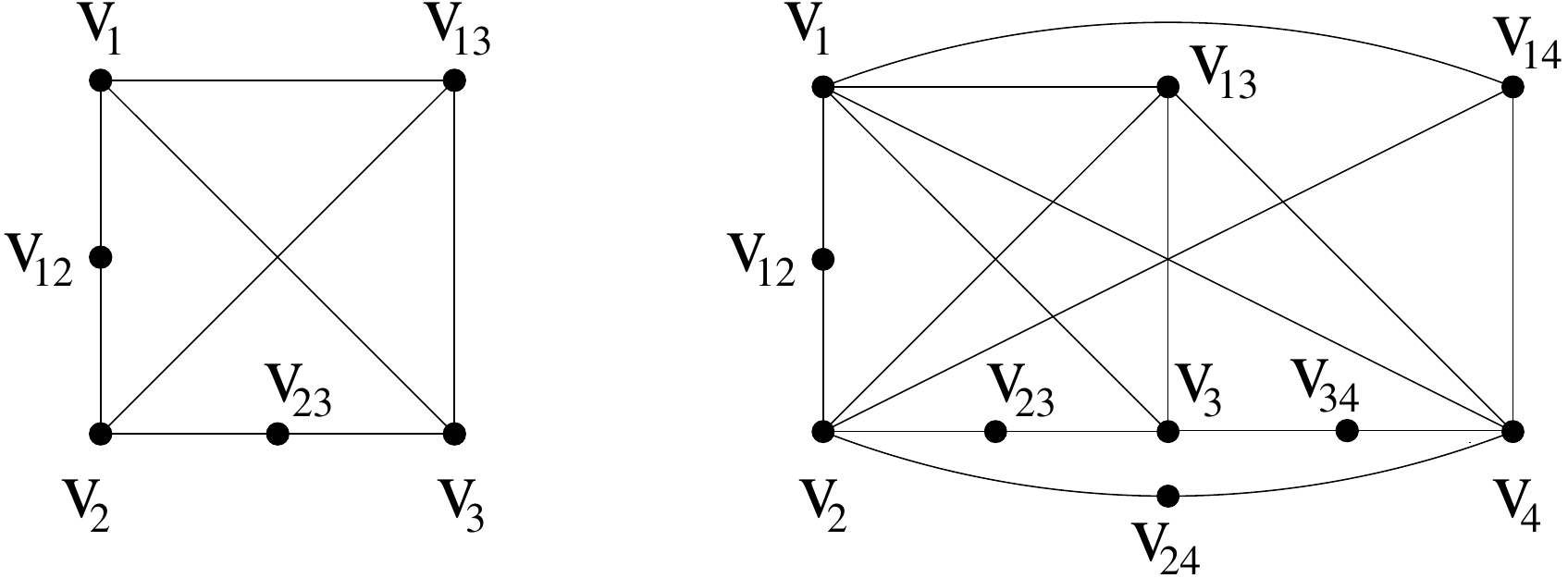}
 \caption{The graphs $H_3$ and $H_4$.}\label{fig:h3h4}
\ec
 \end{figure}The graph $H_3$ is shown in Figure~\ref{fig:h3h4}. It is obtained by taking a complete graph $K_4$, with vertices  $v_1,v_2, v_3$ and $v_{13}$,   and subdividing two adjacent edges: here we insert node $v_{12}$ between $v_1$ and $v_2$ and node $v_{23}$ between nodes $v_2$ and~$v_3$. 

\begin{lemma}\label{lem:H3}
$\egd(H_3)=3$ and  $H_3$ is a minimal forbidden minor for $\GG_2$.
\end{lemma}

\begin{proof}
As $H_3$ has 6 nodes,  $\egd(H_3)\le 3$.  To show equality, we use the following vector labeling for the nodes of $H_3$:
Label the nodes $v_1,v_2,v_3$ by the standard unit vectors $e_1,e_2,e_3\in \oR^3$ and $v_{ij}$ by $(e_i+e_j)/\sqrt 2$ for $1\le i<j\le 3$.
Let $X\in \EE_6$ be their Gram matrix and set $x=\pi_{E(H_3)}(X)\in \EE(H_3)$.
Then $X$ has rank 3 and  $X$ is an extreme point of $\EE_6$. We now show that $X$ is the unique completion of $x$ in $\EE_6$.
For this let $Y\in \fib(x)$.
Consider its principal submatrices  $Z,Z'$ indexed by  $\{v_1,v_2,v_3,v_{13}\}$ and $\{v_1,v_2,v_{12}\}$, of  the form:
$$ Z=\left(\begin{matrix}
1 & a &0 & \sqrt{2}/2\cr
a& 1 & b  & 0\cr
0 &b  & 1& \sqrt{2}/2\cr
\sqrt{2}/2 & 0& \sqrt{2}/2& 1
\end{matrix}\right)\qquad
Z'=\left(\begin{matrix}
1 & a & \sqrt{2}/2\cr
a & 1  & \sqrt{2}/2\cr
\sqrt{2}/2 & \sqrt{2}/2 & 1
\end{matrix}\right),
$$
where $a,b\in\oR$.  Then, $\det(Z)=-(a+b)^2/2$  implies $a+b=0$, and  $\det(Z')=a(1-a)$ implies $a\ge0$. Similarly, $b\ge 0$ using the principal submatrix of $Y$ indexed by $\{v_2,v_3,v_{23}\}$. This shows $a=b=0$ and thus
%
 the entries of $Y$ at the positions $\{v_1,v_2\}$ and $\{v_2,v_3\}$ are uniquely specified. Remains to show that the entries are uniquely specified at the non-edges containing $v_{12}$ or $v_{23}$. 
For this we use Lemma \ref{lemforce}: First the clique $\{v_2,v_3,v_{23}\}$  forces the pairs $\{v_1,v_{23}\}$ and $\{v_{13},v_{23}\}$ and then the clique 
 $\{v_1,v_2,v_{12}\}$ forces the pairs $\{v_{23},v_{12}\}$,  $\{v_{13},v_{12}\}$, and $\{v_3,v_{12}\}$.
  Thus we have shown $Y=X$, which concludes the proof that $\egd(H_3)=3$.

We now verify that it is a minimal forbidden minor. Contracting any edge results in a graph on 5 nodes and we are done by~\eqref{eq:upperbound}. Lastly, we verify that  
$\egd(H_3\backslash  e)\le 2$ for each edge  $e\in E(H_3)$. If deleting the edge $e$  creates a cut node, then the result follows using Lemma \ref{lemcliquesum}. 
Otherwise,  $H_3 \backslash e$ is contained in $T\stp K_2$, where $T$ is a path  (for $e= \{v_2,v_{13}\}$) or a claw $K_{1,3}$ (for  $e=\{v_1,v_{13}\}$ or $\{v_3,v_{13}\}$), and the result follows from Theorem \ref{theomainr}.
\end{proof}

We now describe the graph $H_r$, or rather a class $\HH_r$ of such graphs.
Any graph $H_r\in \HH_r$ is constructed in the following way.
Its node set is $V=V_0\cup V_3\cup \ldots \cup V_r$, where 
$V_0=\{v_{ij}: 3\le i<j\le r\}$ and, for $i\in \{3,\ldots,r\}$,
$V_i=\{v_1,v_2,v_{12}, v_i,v_{1i}, v_{2i}\}$. So $H_r$ has $n={r+1\choose 2}$ nodes.
Its edge set is defined as follows: On each set $V_i$ we put a copy of $H_3$ 
 (with index $i$ playing the role of index 3 in the description of $H_3$ above) and, for each $3\le i<j\le r$, we have the edges $\{v_i,v_{ij}\}$ and $\{v_j,v_{ij}\}$ as well as exactly one edge, call it $e_{ij}$, from the~set
\begin{equation}\label{setF}
F_{ij}=\{\{v_i,v_j\},\{v_i,v_{1j}\}, \{v_j,v_{1i}\},\{v_{1i},v_{1j}\}\}.
\end{equation}
Figure \ref{fig:h3h4} shows the graph $H_4$ for the choice $e_{34}=\{v_4,v_{13}\}$.

\begin{theorem}
For any graph $H_r\in \HH_r$ ($r\ge 3$), $\egd(H_r)=r$.
\end{theorem}

\begin{proof}We label the nodes $v_1,\ldots,v_r$ by  $e_1,\ldots,e_r\in \oR^r$ and  $v_{ij}$ by $(e_i+e_j)/\sqrt 2$. 
Let $X\in \EE_n$ be their Gram matrix and  $x=\pi_{E(H_r)}(X)\in \EE(H_r)$.
Then $X$ is an extreme point of $\EE_n$, 
we show that $\fib(x)=\{X\}$. For this let $Y\in \fib(x)$. We already know that $Y[V_i]=X[V_i]$  for each $i\in \{3,\ldots,r\}$. 
Indeed, as the subgraph of $H_r$ induced by $V_i$ is $H_3$, this  follows from the way we have chosen the labeling and from the proof of Lemma \ref{lem:H3}. Hence we may now assume that we have  a complete graph on each $V_i$ and it 
remains to show that the entries of $Y$ are uniquely specified at the non-edges that are not contained in some set $V_i$ ($3\le i\le r$).
For this note that   the vectors labeling the set $C_{ij}=\{v_i,v_j,v_{ij}\}$ are minimally linearly dependent.
Using Lemma \ref{lemforce}, one can verify that 
%
all remaining non-edges are forced using these sets $C_{ij}$ and thus $Y=X$.  This shows that $\egd(H_r)\ge r$. \end{proof}

In contrast to the graphs  $F_r$ and $G_r$,  we do not know whether $H_r\in \HH_r$ is a {\em minimal} forbidden minor for $\GG_{r-1}$ for $r\ge 4$.

\subsection{Two special graphs: $K_{3,3}$ and $K_5$ }\label{secK33}
 
In this section we consider the graphs $K_{3,3}$ and $K_5$ which will play a special role in the characterization of the class $\GG_2$. 
First we compute the extreme Gram dimension of  $K_{3,3}$. Note that  its  Gram dimension is $\gd(K_{3,3})=4$ as $K_{3,3}$ contains a $K_4$-minor but it contains no $K_5$ and $K_{2,2,2}$-minor; cf. Theorem \ref{thm:gd}.

Our main goal in this section is to show that  the extreme Gram dimension of the graph $K_{3,3}$ is equal to 2, i.e., for any  $x \in \ext \EE(K_{3,3})$ there exists a psd  completion of rank at most 2.
We start by showing that any completion of an element of $\ext \EE(K_{3,3})$ has rank at most 3.

\begin{lemma}\label{lemK33a}
For  $x\in \ext \EE(K_{3,3})$,  any $X\in \fib(x)$  has rank at most 3.
\end{lemma}

\begin{proof}Let $x\in \ext \EE(K_{3,3})$ and let $X\in \fib(x)$ with $\rankspace X\ge 4$.
Let $u_1,\ldots,u_6$ be a Gram representation of $X$ and choose a subset  $\{u_i: i\in I\}$ 
 of linearly independent vectors with $|I|=4$. Let $E_I$ denote the set of edges of $K_{3,3}$ induced by $I$  and set
 $$\UU_I=\{U_{ii}: i\in I\}\cup\{U_{ij}: \{i,j\}\in  E_I\}.$$Then $\UU_I$  consists of linearly independent elements; cf. Lemma \ref{sdrtghui}. By Lemma \ref{lemextG1}, 
$\UU_I$ is contained  in  $\{ \UU_{ii} : i \in [6]\}$ and thus it has dimension at most 6. On the other hand, 
 as any four nodes  induce at least three edges in $K_{3,3}$, we have that $|\UU_I|\ge 4+3=7$ and thus the dimension of $\UU_I$ is at least 7, a contradiction.
\end{proof}

The proof of the main theorem relies on the following two lemmas.


\begin{lemma}\label{lempsd}
Let $X,Z\in \SSS^n$  with $X\succeq 0$ and satisfying:
\begin{equation}\label{eqXz}
Xz=0 \Longrightarrow z^\sfT Zz\ge 0,\ \
Xz=0, z^\sfT Zz=0 \Longrightarrow Zz=0.
\end{equation}
Then $X+tZ\succeq 0$ for some $t>0$.
\end{lemma}

\begin{proof}Up to an orthogonal transformation we may assume  $X=\left(\begin{matrix} D & 0 \cr 0 & 0\end{matrix}\right)$, where $D$ is a diagonal matrix with positive diagonal entries. Correspondingly, write $Z$ in block form:  $Z=\left(\begin{matrix} A & B \cr B^\sfT & C\end{matrix}\right)$.
The conditions (\ref{eqXz}) show that $C \succeq 0$ and that the kernel of $C$ is contained in the kernel 
of $B$. This  implies that $X+t Z \succeq 0$ for some $t>0$.
\end{proof}

\begin{lemma}\label{lemK33b}
Let $x\in \ext \EE(K_{3,3})$, let $X\in \ext  \fib(x)$ with $\rankspace X =3$   and  with Gram representation $\{u_1,\ldots,u_6\}\subseteq \oR^3$. 
Let $V_1=\{1,2,3\}$ and $V_2=\{4,5,6\}$ be the bipartition of the node set of $K_{3,3}$.
Then, there exist matrices $Y_1,Y_2\in \SSS^3$ such that 
$Y_1+Y_2\succ 0$ and  
$$\la Y_k, U_{ii}\ra =0 \ \forall i\in V_k \  \forall  k\in \{1,2\} \ \text{ and }\
\exists k\in \{1,2\} \ \exists i,j\in V_k \ \la Y_k,U_{ij}\ra \ne 0.$$
\end{lemma}

\begin{proof}Define  $\UU_k=\la U_{ii}: i\in V_k\ra\subseteq \WW_k=\la U_{ij}: i,j\in V_k\ra \subseteq \SSS^3 $ for $k=1,2$. 
With this notation we are looking for two matrices $Y_1, Y_2$ such that $Y_1+Y_2 \succ 0, Y_1 \in \UU_1^\perp, Y_2 \in \UU_2^\perp$  and either $Y_1 \not \in \WW_1^\perp$ or $Y_2 \not \in \WW_2^\perp$.

Since $x\in \ext \EE(K_{3,3})$ by Lemma \ref{lemextG2} it follows that $\fib(x)$ is a face of $\EE_6$ and by \eqref{extremepoints} we have that 
$X \in \ext \EE_6$. Then \eqref{eqdimextreme0} implies that
$\dim \la U_{ii} : i \in [6]\ra=6$ and  thus $\dim \UU_1=\dim \UU_2=3$. This implies that $\UU_1 \cap \UU_2=\{0\}$ 
 and thus   $\UU_1^\perp \cup \UU_2^\perp=\SSS^3$. Moreover, as $\dim \UU_1^\perp=\dim \UU_2^\perp=3$ it follows that   $\UU_1^\perp\cap \UU_2^\perp  =\{0\}$ and thus  $\SSS^3 =\UU_1^\perp \oplus \UU_2^\perp$. Lastly,  we have that $\WW_k^\perp \subseteq \UU_k^\perp \ (k=1,2)$ and thus $\WW_1^\perp\cap \WW_2^\perp \subseteq \UU_1^\perp\cap \UU_2^\perp=\{0\}.$

Assume  for contradiction that  $\SSS^3_{++}$ is  contained in 
$\WW_1^\perp \oplus  \WW_2^\perp$. This implies that   
\begin{equation}\label{23edr45}
\WW_1^\perp \oplus \WW_2^\perp= \SSS^3 = \UU_1^\perp\oplus \UU_2^\perp
\end{equation}
and thus 
$\WW_k=\UU_k\ (k=1,2)$. Indeed, \eqref{23edr45} implies that $\dim \WW_1^\perp+\dim \WW_2^\perp=\dim \UU_1^\perp+\dim \UU_2^\perp$ which combined with the fact that $\WW_k^\perp \subseteq \UU_k^\perp\ (k=1,2)$ gives that $\dim \WW_k^\perp=\dim \UU_k^\perp \ (k=1,2)$.  Lastly, using the fact that $\UU_k \subseteq W_k\ (k=1,2)$ the claim follows. 
In turn this implies that  $\WW_1\cap\WW_2=\UU_1\cap \UU_2=\{0\}$.

As $\dim \UU_k=3$, we have $\dim \la u_i :  i\in V_k\ra\ge 2$ for $k=1,2$. Say,  $\{u_1, u_2\}$ and $\{u_4,u_5\}$ are linearly independent.  As $\dim \la u_i: i \in [6]\ra=3$, there exists a nonzero vector $\lambda \in \oR^4$ such that $0\neq w=\lambda_1 u_1 +\lambda_2 u_2=\lambda_3 u_4 +\lambda_4 u_5$. Notice that the scalar $w$ is nonzero for otherwise the vectors $u_1, u_2$ would be dependent.   Hence  we obtain that $ww^\sfT\in \WW_1\cap \WW_2$,   contradicting the fact that $\WW_1\cap\WW_2=\{0\}$.


 Hence we have shown that $\SSS^3_{++}\not\subseteq \WW_1^\perp \oplus \WW_2^\perp$.
So  there exists a positive definite matrix $Y$ which does not belong to
 $\WW_1^\perp\oplus \WW_2^\perp$.
 Write $Y=Y_1+Y_2$, where $Y_k\in \UU_k^\perp$ for $k=1,2$. 
We may assume, say, that $Y_1\not\in \WW_1^\perp$.  
Thus $Y_1,Y_2$ satisfy the  lemma.
\end{proof}

We can now show  the main result of this section. 

\begin{theorem}\label{theoK33}For the graph $K_{3,3}$ we have that $\egd(K_{3,3})=2$, i.e., for any partial matrix  $x \in \ext \EE(K_{3,3})$ there exists a completion $X \in  \fib(x)$ with $\rankspace X\le 2$. 
\end{theorem}
\begin{proof}
Let $x\in \ext \EE(K_{3,3})$ and  let $X\in \fib(x)$ be an extreme point of $\EE_6$ (which exists by Lemma \ref{lemextG2} (ii)). 
We assume that $\rank X=3$ (else we are done).
Let $\{u_1,\ldots,u_6\}\subseteq \oR^3$ be a Gram representation of $X$ and let $Y_1$ and $Y_2$ be the matrices provided by Lemma \ref{lemK33b}. 
Moreover,  define the matrix $Z\in \SSS^6$ by 
$Z_{ij}= \lan Y_k,U_{ij}\ra$ for $i,j\in V_k,$ $k\in \{1,2\}$, and $Z_{ij}=0$ for $i\in V_1$, $j\in V_2$.  
By Lemma~\ref{lemK33b},  $Z$ is a nonzero matrix with zero diagonal entries and with zeros at the positions corresponding to  the edges of $K_{3,3}$.

Next we show  that $X+t Z\succeq 0$ for some $t>0$, using Lemma \ref{lempsd}.
For this 
 it is enough to verify  that (\ref{eqXz}) holds. 
Assume $Xz=0$, i.e., $a:=\sum_{i\in V_1}z_iu_i=-\sum_{j\in V_2} z_ju_j$.
Then, $$z^\sfT Zz= \sum_{k=1,2}\sum_{i,j\in V_k}z_iz_j \lan Y_k,U_{ij}\ra =   \lan Y_1+Y_2,aa^\sfT\ra \ge 0,$$ since $Y_1+Y_2\succ 0$.
Moreover, $z^\sfT Zz=0$ implies $a=0$ and thus
$Zz=0$ since, for $i\in V_k$,  $(Zz)_i =\sum_{j\in V_k} \lan Y_k, U_{ij}\ra z_j= \pm \lan Y_k, (u_ia^\sfT+au_i^\sfT)/2\ra$.
Hence, the matrix $X+tZ$ also belongs to the fiber of $x$.  This shows that  $|\fib(x)|\ge 2$.
As, by Lemma \ref{lemK33a}, all matrices in $\fib(x)$ have rank at most 3, it follows that $\fib(x)$ contains a matrix of rank at most 2 (indeed, any matrix in the relative interior of $\fib(x)$ has rank 3 and any matrix on the boundary has rank at most 2).
\end{proof}

We now know that  both graphs $K_{3,3}$ and $K_5$ belong to the class $\GG_2$. We next show that they are in some sense maximal for this property.

\begin{lemma}\label{lemblock}
Let  $G$ be a 2-connected graph that contains $K_5$ or   $K_{3,3}$ as a proper subgraph. Then, $G$ contains  $H_3$ as a minor and thus~$\egd(G)\ge 3$.
\end{lemma}

\begin{proof}The proof is based on the following observations. If $G$ is a 2-connected graph containing $K_5$ or $K_{3,3}$ as a proper  subgraph, 
then $G$ has a minor $H$ which is one of the following  graphs: (a) $H$ is $K_5$ with one more node adjacent to two nodes of $K_5$, 
(b) $H$ is $K_{3,3}$ with one more edge added, (c) $H$ is $K_{3,3}$ with one more node adjacent to two adjacent nodes of $K_{3,3}$.
Then $H$ contains a $H_3$ subgraph in cases (a) and (b), and a $H_3$ minor in case (c) (easy verification).
Hence, $\egd(G)\ge \egd(H_3)=3$.
\end{proof}

We conclude this section with a lemma that will be used in the proof of  Theorem~\ref{theofree}.

\begin{lemma}\label{lemnewfwuef} Let $G$ be a 2-connected graph with $n \ge 6$ nodes. Then, 
\begin{itemize}
 \item[(i)] If $G$ has no $F_3$-minor then $\omega(G)\le 4$.
\item[(ii)] If $G$ is chordal and has no $F_3$-subgraph then $\omega(G)\le 4$.
\end{itemize}
\end{lemma}
\begin{proof} Assume for contradiction that $\omega(G) \ge 5$ and let $U\subseteq V$ with  $G[U]=K_5$.   Since $G$ is 2-connected and $n\ge 6$, there exists a node $ u \not \in U$ which is connected by two vertex disjoint paths to two distinct  nodes  $v,w \in U$; let  $P_{uv}$ and $ P_{uw}$ be such  shortest  paths.
In case $(i)$, contract the paths $P_{uv}$ and $P_{uw}$ to get  a node adjacent to both  $v$ and $w$. Then,   we can easily see that  $G$ has  an $F_3$-minor, a contradiction.
In case $(ii)$, let  $v'\in P_{uv}$ and $w' \in P_{uw}$  with $(v,v'), (w,w') \in E(G)$. Since $G$ is chordal and the paths are the shortest possible, at least one of the edges $(v,w')$ or $(w,v')$ will be present in $G$. This implies that   $G$ contains an $F_3$-subgraph, a contradiction. 
\end{proof}


\section{Graphs with extreme Gram dimension at most~2} \label{secmainG2}

In this section we characterize the class $\GG_2$ of  graphs with extreme Gram dimension at most  2.
Our main result  is the following:

\begin{theorem}\label{thm:main} For any graph $G$,

$$\egd(G)\le 2 \text{ if and only if } G \text{ has no minors } F_3 \text{ or } H_3.$$
\end{theorem}
\noindent The graphs $F_3$ and $H_3$ are illustrated in Figure  \ref{adouguowef} below. 

 \begin{figure}[h]
\bc
  \includegraphics[scale=0.4]{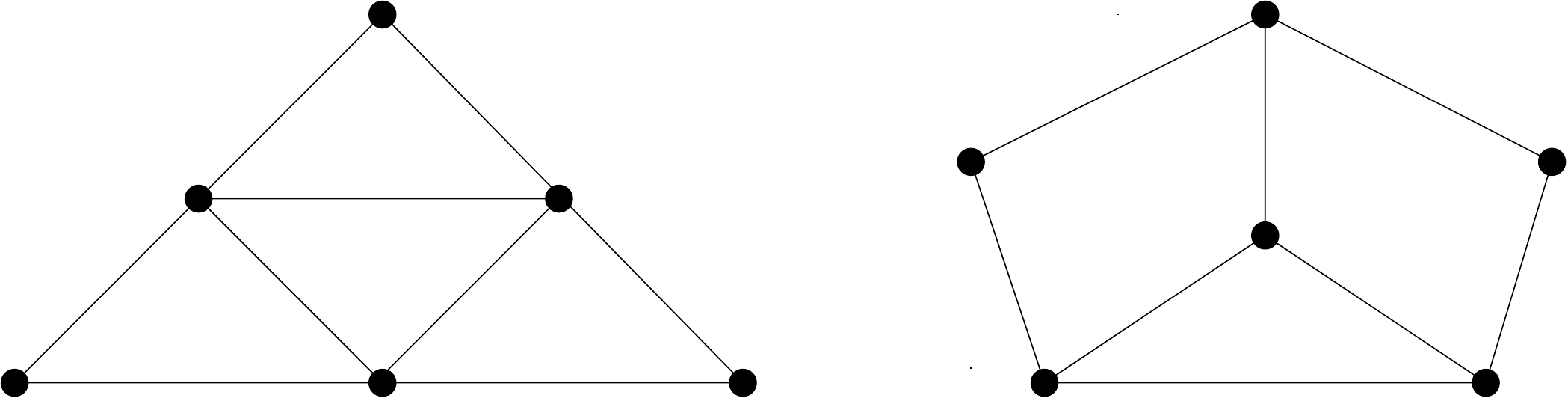}
  \caption{The graphs $F_3$ and $H_3$}
  \label{adouguowef}
\ec
\end{figure} 
In the previous sections we  established  that the graphs $F_3$ and $H_3$ are minimal forbidden minors for the class of graphs satisfying $\egd(G)\le 2$. In order to prove Theorem \ref{thm:main} it remains to show that a graph $G$ having  no $F_3$ and $H_3$ minors satisfies  $\egd(G)\le 2$.

By Lemma \ref{lemcliquesum} we may assume that $G$ is 2-connected. Moreover, we may assume that  $|V(G)|\ge 6$, since for graphs on at most  five nodes  we know that $\egd(G)\le\egd(K_5)= 2$ (recall \eqref{eq:upperbound}). Additionally, since $\egd(K_{3,3})=2$ (recall Theorem \ref{theoK33}) we may also assume that $G\ne K_{3,3}$.

Consequently,  it suffices to consider  2-connected graphs with at least 6 nodes that are different from  $K_{3,3}$. Then, 
Theorem \ref{thm:main}   follows  from  the equivalence of the first two items in  the next theorem.

\begin{theorem} \label{theomain}
Let $G$ be a 2-connected graph with  $n\ge 6$ nodes and   $G\ne K_{3,3}$. Then, the following assertions are equivalent.
\begin{itemize}
\item[(i)]
$\egd(G)\le 2$.
\item [(ii)] 
$G$ has no minors $F_3$ or $H_3$.
\item[(iii)] 
$\sla(G)\le 2$, i.e., $G$ is a minor  of  $T\stp K_2$ for some tree $T$.
\end{itemize}
\end{theorem}

The implication $(i) \Longrightarrow (ii)$ follows from Theorem \ref{theoFr} and Theorem \ref{lem:H3}. Moreover, the implication $(iii) \Longrightarrow (i) $  follows from Theorem \ref{theomainr}. The rest  of this  chapter is dedicated to  proving the  implication $(ii) \Longrightarrow (iii).$
The proof   consists of  two steps. First we consider the chordal case and show:

 \begin{itemize}
\item[(1)]
 {\bf The chordal case:} Let $G$ be a 2-connected  chordal graph with $n\ge 6$ nodes.  Then,  $G$ has no $F_3$ or $H_3$-minors    if and only if  $G$ is a contraction minor of $T\stp K_2$, for some tree $T$ (Section \ref{secchordal}, Theorem \ref{theofree}). 
\end{itemize}
Then, we reduce the general case to the chordal case and show: 
\begin{itemize}
\item[(2)]
  {\bf Reduction to the chordal case:} Let $G$ be a 2-connected graph with $n \ge 6$ nodes and   $G\ne K_{3,3}$.  If   $G$ has no $F_3$ or $H_3$-minors   then $G$  
 is  a subgraph of a chordal graph  with  no $F_3$ or $H_3$-minors.  
\end{itemize}

Notice that, in  case (1),  $G$ is by assumption chordal and thus  the case $G=K_{3,3}$ is automatically excluded. For case (2), we first need to exclude $K_4$ instead of $H_3$ (Section \ref{secnoK4}, Theorem \ref{theonoF3K4}) and then we derive from this special case  the general result
(Section \ref{secnoH3}, Theorem \ref{theonoF3H3}).


\subsection{The  chordal case}\label{secchordal}

Our  goal in  this section is to    characterize the  $2$-connected chordal graphs $G$ with  $\egd(G)\le 2$. 
By Lemma \ref{lemblock}, if $G\ne K_5$ has $\egd(G)\le 2$, then $\omega(G)\le 4$.  
Throughout this section we denote by  $\CC$ the family of all 2-connected chordal graphs with $\omega(G)\le 4$.
Any graph $G\in \CC$ is a clique 2-  or 3-sum of  $K_3$'s and $K_4$'s.
Note that  $F_3$ belongs to $\CC$ and has $\egd(F_3)=3$. On the other hand, any graph  $G=T\stp K_2$ where $T$ is a tree, belongs to $\CC$ and has $\egd(G)=2$. These graphs are ``special clique 2-sums" of $K_4$'s, as they satisfy the following property: 
every 4-clique  has at most two edges which are cutsets 
and these two edges are not adjacent. This motivates the following definitions, useful in the proof of Theorem \ref{theofree} below.  

\begin{definition}\label{def_free}
Let $G$ be a  2-connected  chordal graph with $\omega(G)\le 4$.
\begin{itemize}
\item[(i)] An edge of $G$  is  called  {\em free} if it  belongs to exactly one 
maximal clique  and {\em non-free} otherwise. 
\item[(ii)] A 3-clique  in $G$ is  called  {\em free} if it contains at least one free edge.
\item[(iii)] A  4-clique  in $G$ is called {\em free} if it does not have two adjacent non-free edges. A free 4-clique   can be partitioned as $\{a,b\} \cup \{c,d\}$, called its {\em two sides}, where only  $\{a,b\}$ and $\{c,d\}$ can be non-free.
\item[(iv)] $G$ is  called free if all its maximal cliques are free.
\end{itemize}

\end{definition}



For instance, $F_3$, $K_5\setminus e$  
 (the clique 3-sum of two $K_4$'s) are  not  free. Hence any free  graph in $\CC$ is a clique  2-sum of  free $K_3$'s and free $K_4$'s. 
 Note also that $\sla(K_5\setminus e)=3$. 
 We  now show that for  a graph  $G \in \CC$ the property of being free is equivalent to  having $\sla(G)\le 2$ and also to having  $\egd(G)\le 2$.

\begin{theorem}\label{theofree}
Let $G$ be a 2-connected chordal graph with $n\ge 6$ nodes. The following assertions are equivalent:
\begin{itemize}
\item[(i)] $G$ has no minors $F_3$ or $H_3$.
\item[(ii)] $G$ does not contain $F_3$ as a subgraph.
\item[(iii)]  $\omega(G)\le 4$ and $G$ is free. 
\item[(iv)] $G$ is a contraction minor of  $T\stp K_2$ for some tree $T$.
\item[(v)] $\sla(G)\le 2.$
\item[(vi)] $\egd(G)\le 2$.
\end{itemize}
\end{theorem}

\begin{proof}The implication
 $(i) \Rightarrow (ii)$ is clear and the implications $(iv) \Rightarrow (v) \Rightarrow (vi) \Rightarrow (i)$ follow  from earlier results. 

\noindent $(ii) \Rightarrow (iii)$: Assume that $(ii)$ holds. By  Lemma \ref{lemnewfwuef} $(ii)$ it follows that   $\omega(G)~\le 4$. 
Our first goal is to show that $G$ does not contain clique 3-sums of $K_4$'s, i.e., it does not contain a $K_5 \setminus e$ subgraph. 
For this, assume that 
  $G[U]=K_5\backslash e$ for some $U\subseteq V(G)$.
As  $|V(G)|\ge  6$ and $G$ is 2-connected chordal, there exists  a node $u\not\in U$  which is adjacent  to  two adjacent nodes of   $U$. Then,   one can find a $F_3$ subgraph in $G$, a contradiction. 
Therefore, $G$ is a clique 2-sum of $K_3$'s and $K_4$'s. 
We  now show that each of them  is free.

Suppose first that    $C=\{a,b,c\}$ is a maximal 3-clique which is not free. Then, there exist nodes $u,v,w\not\in C$ such that $\{a,b,u\}$, $\{a,c,v\}$, $\{b,c,w\}$ are cliques in $G$. Moreover, $u,v,w$ are pairwise distinct (if $u=v$ then $C\cup\{u\}$ is a clique, contradicting maximality of $C$) and we find a $F_3$ subgraph in $G$.

Suppose now that $C=\{a,b,c,d\}$ is a 4-clique which is not free and,  say,  both edges $\{a,b\}$ and $\{a,c\}$ are non-free. Then, there exist nodes $u,v\not\in C$ such that $\{a,b,u\}$ and $\{a,c,v\}$ are cliques. Moreover, $u\ne v$ (else we find a $K_5\backslash e$ subgraph) and thus  we find a $F_3$ subgraph in $G$. Thus (iii) holds.

\smallskip
\noindent $(iii) \Rightarrow (iv):$  Assume that $G$ is free, $G\ne K_4,K_3$  (else we are done). 
When all maximal cliques are 4-cliques, it is easy to show using induction on $|V(G)|$ that $G=T\stp K_2$,
 where $T$ is a tree and each side of a 4-clique of $G$  corresponds to a node of $T$.

Assume  now that $G$ has a maximal 3-clique $C=\{a,b,c\}$. Say, $\{b,c\}$ is free and  $\{a,b\}$ is a cutset. 
Write $V(G)=V'\cup V''\cup \{a,b\}$, where $V''$ is the (vertex set of the) component of $G\backslash \{a,b\}$ containing $c$, and $V'$ is the union of the other components.
Now replace node $a$ by two new nodes $a',a''$ and replace $C$ by the 4-clique  $C'=\{a',a'',b,c\}$. Moreover, 
replace each edge $\{u,a\}$ by $\{u,a'\}$ if $u\in V'$ and by $\{u,a''\}$ if $u\in V''$.
Let $G'$ be the graph obtained in this way.  Then $G'\in \CC$  is free, $G'$  has one less maximal 3-clique than $G$, 
and $G=G'\slash \{a',a''\}$. 
Iterating, we obtain a graph $\hG$ which is a clique 2-sum of free $K_4$'s and contains $G$ as a contraction minor. By the above,  $\hG=T\stp K_2$ and thus $G$ is a contraction minor of $T\stp K_2$.
\end{proof}

\subsection{Structure of  the graphs with no $F_3$ and  $K_4$-minor}\label{secnoK4}
 
 In this section we investigate the structure of the graphs with no $F_3$ or $K_4$-minors. 
 We start with two technical lemmas.
 
 \begin{lemma} \label{lem_minor_cut}
Let $G$ and $M$ be two  $2$-connected graphs, let $\{x,y\}\not\in E(G)$ be a cutset in $G$, and let 
 $r\ge 2$ be the number of components of $G\setminus\{x,y\}$.
 \begin{itemize}
\item[(i)] 
Assume that 
 $G\in\mathcal{F}(M)$, but the graph $G+\{x,y\}$ has a $M$-minor with $M$-partition $\{V_i: i\in V(M)\}$. If $x\in V_i$ and  $y\in V_j$,  then   $M\setminus\{i,j\}$ has at least $r\ge 2$ components (and thus  $i\neq j$).   
 \item[(ii)]
 Assume that  $M$ does not have  two adjacent nodes forming  a cutset in $M$.
If $G \in {\mathcal F}(M)$, then   $G+\{x,y\} \in {\mathcal F} (M)$.
\end{itemize}
\end{lemma}

\begin{proof}(i) Let $C_1,\ldots,C_r\subseteq V(G)$  be the node sets of the components of $G\setminus \{x,y\}$.
As $G$ is $2$-connected,  there is an $x-y$  path $P_s$  in $G[C_s\cup \{x,y\}]$ for each $s\in [r]$. Notice that $P_s\ne \{x,y\}$ since $P_s$ is a path in $G$.  Our first goal is to show that  every component of $M\setminus\{i,j\}$ corresponds to exactly one component of $G\setminus\{x,y\}$. For this, let $U$ be a component of $M\setminus\{i,j\}$. By the definition of the $M$-partition,  the graph $G[\bigcup_{k\in U} V_k]$ is connected.
As   $x,y\not\in \bigcup_{k\in U} V_k$,   we deduce that  $\bigcup_{k\in U} V_k \subseteq C_s$ for some $s\in[r]$. We can now conclude the proof. Assume for contradiction that  $M\setminus\{i,j\}$ has less than $r$ components.
Then there is at least one component $C_s$  which does not correspond to any component of $M\setminus\{i,j\}$ which means that 
 $(\bigcup_{k\neq i,j} V_k) \cap C_s = \emptyset$. Indeed, if $V_k\cap C_s$ for some $k\ne i,j$ then  since $C_s$ is a connected component of $G\setminus\{x,y\}$ it follows that $\cup_{\lambda \in U}V_{\lambda} \subseteq C_s$, where 
 $U$ is the component of $M\setminus \{i,j\}$ that contains $k$.  Summarizing we know that  $C_s \subseteq V_i\cup V_j$. Hence the path $P_s$ is contained in $G[V_i\cup V_j]$,  thus 
$\{V_i:i\in V(M)\}$ remains an $M$-partition of $G$ (recall that $P_s\ne \{x,y\}$)  and we find a $M$-minor in $G$, a contradiction.
Therefore,  $M\setminus\{i,j\}$ has at least $r\ge 2$ components. This implies that  $\{i,j\}$ is a cutset of $M$ since $M$ is 2-connected it follows that $i\ne j$. 

(ii) Assume  $G+\{x,y\}$ has a $M$-minor, with corresponding $M$-partition $\{V_i: i\in V(M)\}$.
 By (i), the nodes $x$ and $y$ belong to two distinct classes  $V_i$ and $V_j$ 
 and  $\{i,j\}$ is a cutset in $M$. By the hypothesis, this implies that  $\{i,j\}\not\in E(M)$
and thus  $M$ is a minor of $G$, a contradiction.
\end{proof}

%

%
We continue with a lemma that will be essential for the next theorem.

\begin{lemma}\label{lem3paths}
Let $G\in \FF(K_4)$ be a 2-connected graph and  let $\{x,y\}\not\in E(G)$.
If there are at least three (internally vertex) disjoint paths from $x$ to $y$, then $\{x,y\}$ is a cutset and $G\backslash \{x,y\}$ has at least 3 components.
\end{lemma}

\begin{proof}Let $P_1,P_2,P_3$ be  distinct vertex disjoint paths from $x$ to $y$.  Then $P_1\setminus \{x,y\}$, $P_2\setminus\{x,y\}$ and $P_3\setminus \{x,y\}$ lie in distinct components of $G\setminus \{x,y\}$, for otherwise  $G$  would contain a homeomorph of $K_4$.
\end{proof}

We  now arrive at  the main result of this section.

\begin{theorem}\label{theonoF3K4}
Let  $G\in \FFKc $ be a 2-connected graph on $n\ge 6$ nodes. Then, there exists a chordal graph $Q\in \FF(F_3,K_4)$ containing   $G$ as a subgraph.
\end{theorem}

\begin{proof}Let $G$ be a 2-connected graph in  $\FF(F_3,K_4)$.
As a first step, consider  $\{x,y\}\not\in E(G)$ such that  there exist at least three disjoint paths in $G$  from $x$ to $y$. Then,  Lemma~\ref{lem3paths} implies that $\{x,y\}$ is a cutset of $G$ and $G\setminus\{x,y\}$ has at least three components. As a first step we show  that  we can add the edge $\{x,y\}$ without creating a $K_4$ or $F_3$-minor, i.e., $G+\{x,y\}\in \FF(F_3,K_4)$. 

As $\{x,y\}$ is a cutset, Lemma~\ref{lem_minor_cut} (ii) applied  for $M=K_4$ gives that $G+\{x,y\}$ does not have a $K_4$ minor. Assume for contradiction that $G+\{x,y\}$ has an $F_3$ minor. Again, Lemma~\ref{lem_minor_cut} (i) applied for $M=F_3$ implies that $x\in V_i, y\in V_j,$ where $F_3\setminus \{i,j\}$ has at least 3 components. Clearly there is no such pair of vertices in $F_3$ so we arrived at a contradiction. 
Consequently,  we can add edges iteratively until we obtain  a graph $\hG\in \FF(F_3,K_4)$ containing $G$ as a subgraph and satisfying:
\begin{equation}\label{eq3paths}
\forall \{x,y\}\not\in E(\hG)\  \text{ there are at most two disjoint } x-y\text{ paths  in } \hG.
\end{equation}
If $\hG$ is chordal we are done.
So consider a chordless circuit $C$ in $\hG$. Note that any circuit $C'$ distinct from $C$, which  meets $C$ in at least two nodes, meets $C$ in exactly two nodes (if they meet in at least 3 nodes then we can find a $F_3$ minor) that are adjacent (if they are not adjacent then there exist three internally vertex disjoint paths between them, contradicting  \eqref{eq3paths}).
Call an edge of $C$ {\em busy} if it is contained in some  circuit $C'\ne C$. 
If $e_1 \ne e_2$ are two busy edges of $C$ and $C_i\ne C$ is a circuit containing $e_i$, then $C_1,C_2$ are (internally) disjoint (use (\ref{eq3paths})).
Hence $C$ can have at most two busy edges, for otherwise one would find a $F_3$-minor in $\hG$.

We now show how to triangulate $C$ without creating a $K_4$ or $F_3$-minor: If $C$ has two busy edges denoted, say, $\{v_1,v_2\}$ and $\{v_k,v_{k+1}\}$ (possibly $k=2$), then we add the edges $\{v_1,v_i\}$ for $i\in \{3,\ldots,k\}$ and the edges $\{v_k,v_i\}$ for $i\in \{k+2,\ldots,|C|\}$, see Figure \ref{fig_circ_triang} a). If $C$ has only one busy edge $\{v_1,v_2\}$, add the edges $\{v_1,v_i\}$ for $i\in \{3,\ldots, |C|-1\}$, see Figure \ref{fig_circ_triang} b). 
(If $C$ has no busy edge then $G=C$,  triangulate from any node and we are done).
 \begin{figure}[h!]
\bc
  \includegraphics[scale=0.55]{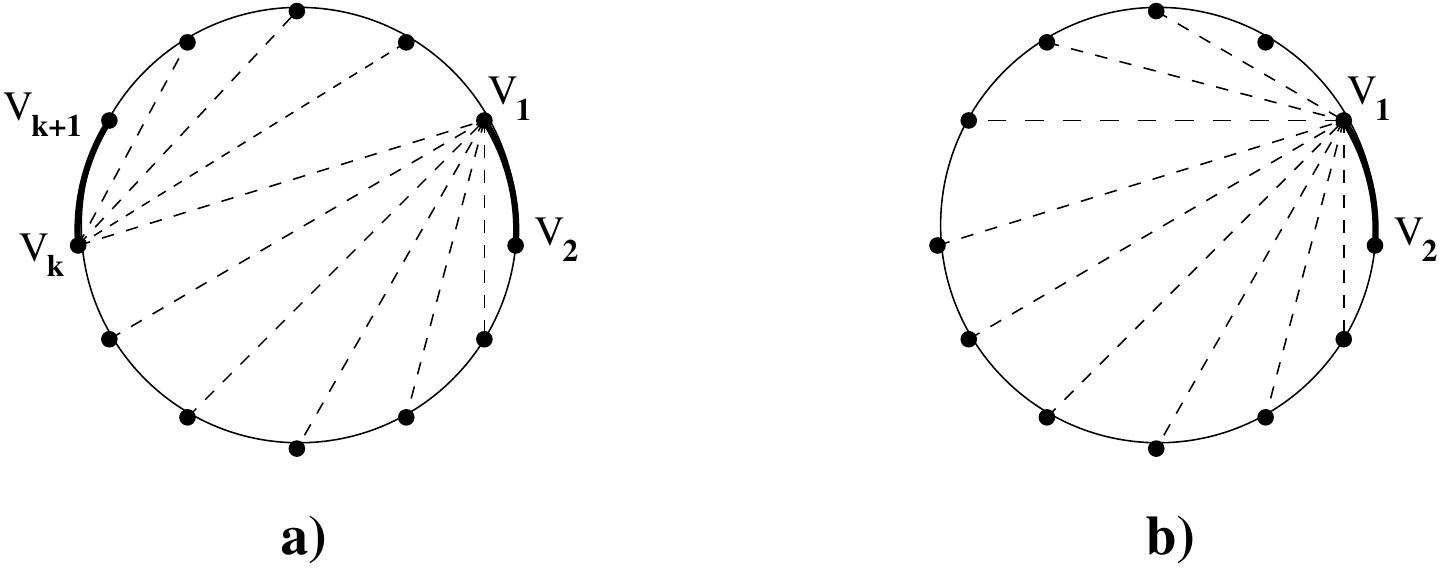}
\caption{Triangulating 
 a chordless circuit with a) two or b) one busy edge.}\label{fig_circ_triang}
\ec
\end{figure}

Let $Q$ denote the graph obtained from $\hG$  by  triangulating all its chordless circuits in this way.
Hence $Q$ is a chordal extension of $\hG$ (and thus of $G$). We show that $Q\in \FF(F_3,K_4)$.
First we see that $Q\in \FF(K_4)$ by applying iteratively Lemma \ref{lem_minor_cut} (ii) (for $M=K_4$): For each $i\in \{3,\ldots,k\}$, $\{v_1,v_i\}$ is a cutset of $\hG$ and of $\hG+\{\{v_1,v_j\}: j\in \{3,\ldots, i-1\}\}$ (and analogously for the other added edges $\{v_k,v_i\}$).
Hence $Q$ is a clique 2-sum of triangles. We now verify that each triangle is free which will conclude the proof, using Theorem \ref{theofree}.

For this let $\{a,b,c\}$ be a triangle in $Q$. First note that if (say) $\{a,b\} \in E(Q)\setminus E(\hG)$, then $a,b,c$ lie on a common chordless circuit $C$ of $\hG$. 
Indeed, let $C$ be a chordless circuit of $\hG$ containing $a,b$ and assume $c\not\in C$.
By (\ref{eq3paths}), $\hG \setminus\{a,b\}$ has at most two components and thus there is a path from $c$ to one of the two paths composing $C\setminus \{a,b\}$. Together with the triangle $\{a,b,c\}$ this gives a homeomorph of $K_4$ in $Q$, contradicting $Q\in \FF(K_4)$, just shown above.
Hence the triangle $\{a,b,c\}$ lies in  $C$ and thus has a free edge.

Suppose now that $\{a,b,c\}$ is a triangle contained in $\hG$.  
If it is not free then  there is a $F_3$ on $\{a,b,c,x,y,z\}$ where $x$ (resp., $y$, and $z$) is adjacent to $a,b$ (resp., $a,c$, and $b,c$).
Say $\{x,a\}\in E(Q)\setminus E(\hG)$ (as there is no $F_3$ in $\hG$). Then  $x,a,b$ lie on a chordless circuit $C$ of $\hG$ and $\{x,b\}\in E(\hG)$ (since $\{a,b\}$ is a busy edge). 
Moreover, $c,y,z\not\in C$ for otherwise we get a $K_4$-minor in $Q$.
Then  delete the edge $\{x,a\}$ and replace it by the $\{x,a\}$-path along $C$. Do the same for any other edge of $E(Q)\setminus E(\hG)$  connecting $y$ and $z$ to $\{a,b,c\}$. After that we get a $F_3$-minor in $\hG$, a contradiction.
\end{proof}

\subsection{Structure of  the graphs with no $F_3$ and   $H_3$-minor}\label{secnoH3}

Here   we investigate the graphs  $G$ in the class $ \FF(F_3,H_3)$.  By the results in Section \ref{secnoK4} we may assume that $G$   contains some homeomorph of $K_4$. 
 Figure~\ref{fig_homK4}  shows a homeomorph of $K_4$, where the original nodes are denoted as 1,2,3,4 and called its {\em corners}, and   the wiggled lines correspond to  subdivided edges (i.e., to paths $P_{ij}$ between the corners $i\ne j\in [4]$).
 \begin{figure}[!h]
\bc  \includegraphics[scale=0.6]{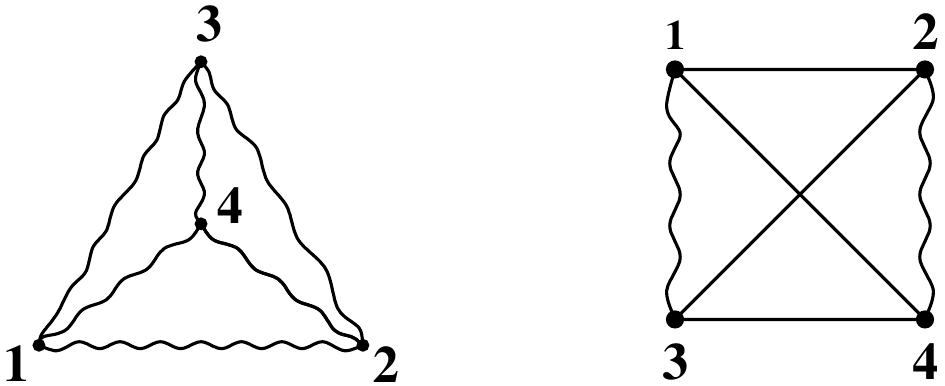}
 \caption{A homeomorph of $K_4$ and its two sides (cf. Lemma \ref{lem2sides})}\label{fig_homK4}
\ec
 \end{figure}
 
To help the reader visualize $F_3$ and $H_3$  we use  Figure \ref{fig_defF6G6}. Notice the special role of node $5$ in $H_3$  (denoted by a square) and of the (dashed) triangle $\{1,2,3\}$. 

\begin{figure}[h!]
\bc  \includegraphics[scale=0.7]{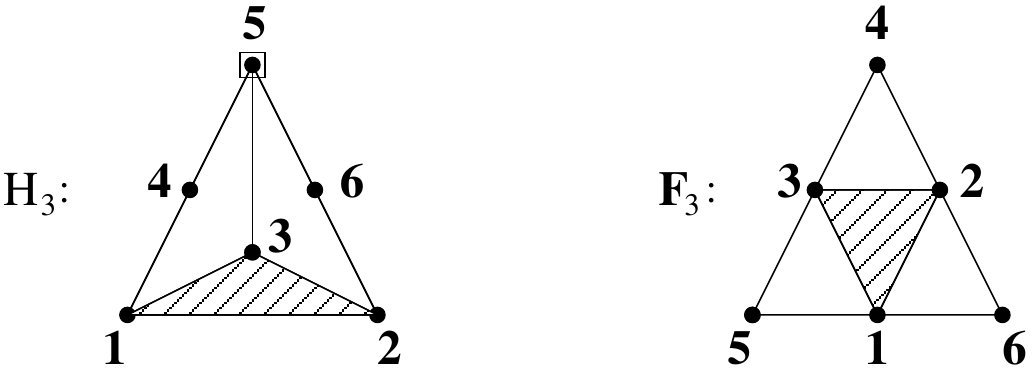}
\caption{The graphs $H_3$ and $F_3$.}\label{fig_defF6G6}
\ec
\end{figure}
 
\medskip
The starting point of the proof is to investigate the structure of homeomorphs of $K_4$ in a graph of $\FF(H_3)$.

\begin{lemma}\label{lem2sides}
Let $G$ be a 2-connected graph in $\mathcal{F}(H_3)$ on $n\ge 6$ nodes and let $H$ be a homeomorph of $K_4$ contained  in $G$.
Then there is a partition of the corner nodes of $H$ into $\{1,3\}$ and $\{2,4\}$ for which the following holds.
\begin{itemize}
\item[(i)] Only the paths $P_{13}$ and $P_{24}$ can have more than two nodes.
\item[(ii)] 
 Every component of $G\setminus H$ is connected to $P_{13}$ or to $P_{24}$.
 \end{itemize}
 Then $P_{13}$ and $P_{24}$ are called the two {\em sides} of $H$ (cf. Figure \ref{fig_homK4}).
\end{lemma}


\begin{proof} We use the graphs from Figure \ref{fig_F6G6comp_minors} which all contain a subgraph $H_3$. \\
{\bf Case 1:} $H=K_4$. 
If $G\setminus H$ has a unique  component $C$ then $|C|\ge 2$ as $n\ge 6$.  
   If $C$  is connected to two nodes of $H$, then the conclusion of the lemma holds.
  Otherwise,  $C$ is connected to at least three nodes in $H$ and then the graph from Figure~\ref{fig_F6G6comp_minors}\,a) is a minor of $G$, a contradiction.

If there are at least two components in $G\setminus H$, then they cannot be connected to two adjacent edges of $H$  for, otherwise,
 the graph of Figure \ref{fig_F6G6comp_minors}\,b) is a minor of $G$, a contradiction. Hence the lemma holds.

\noindent
{\bf Case 2:}  $H\neq K_4$. 
Say, $P_{13}$ has at least 3 nodes. Then  the edges $\{1,i\}$, $\{3,i\}$ ($i=2,4$) cannot be subdivided (else $H$ is a homeomorph of $H_3$). So (i) holds. We now show (ii).
Indeed,  if a component of $G\setminus H$ is connected to both $P_{13}$ and $P_{24}$, then at least one of the graphs in Figure \ref{fig_F6G6comp_minors} c) and d) will be a minor of $G$, a contradiction. 
\end{proof}
\begin{figure}[h!]
\bc  \includegraphics[scale=0.65]{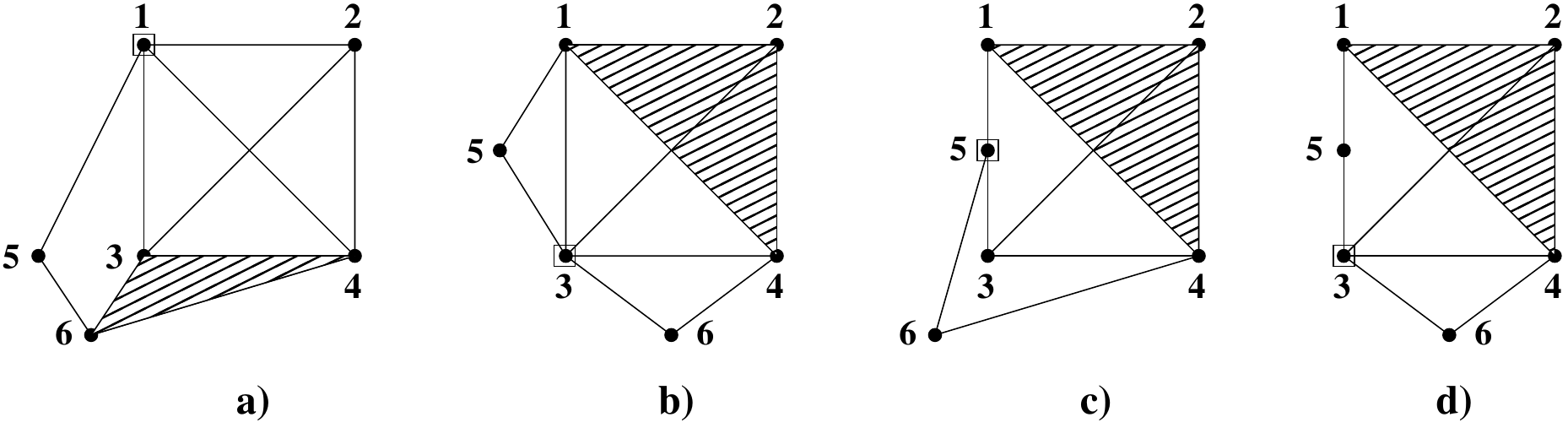}
\caption{Bad subgraphs  in the proof of Lemma \ref{lem2sides}.}\label{fig_F6G6comp_minors}
\ec
\end{figure}

Lemma \ref{lem2sides} implies that there is no path with at least 3 nodes  between the sides of a $K_4$-homeomorph. We now show that, moreover,
 there is no additional edge between the two sides. More precisely: 
\begin{lemma}\label{lem_K4ind}
Let  $G\ne K_{3,3}$ be a 2-connected graph in $\mathcal{F}(H_3)$ on  $n\ge 6$ nodes and let  $H$ be a homeomorph of $K_4$ contained in $G$.
Then  there exists  no edge between the two sides of $H$ except between their endpoints.
\end{lemma}

\begin{proof}
Say, $P_{13}$ and $P_{24}$ are the two sides of $H$. Assume for a contradiction that $\{a,b\}\in E(G)$, where $a$ lies on $P_{13}$ and $b$ on $P_{24}$.

Assume first  that $a$ is an internal node of $P_{13}$ and $b$ is an internal node of $P_{24}$.
If $|V(H)|=6$, then  $H=K_{3,3}$ and   Lemma \ref{lemblock} implies that $G$ has a $H_3$ minor,  a contradiction. Hence,  $|V(H)|> 6 $  and we can assume 
w.l.o.g.~that the path from $1$ to $a$ within $P_{13}$ has at least 3 nodes. 
Then $G$ contains a homeomorph of $K_4$ with corner nodes $1,b,4,a$, where the two paths from $1$ to $a$ and from $1$ to $b$ (via $2$) have at least 3 nodes,  giving  
 a $H_3$ minor and thus a contradiction. 

Assume now that only $a$ is an internal node of $P_{13}$ and, say  $b=2$.
If  $|V(H)|=5$, then $G\setminus H$ has at least one component. By Lemma \ref{lem2sides}, this component  connects either to 
the path $P_{13}$ or to the edge $\{2,4\}$. In both cases, it is easy to verify  that one of  the graphs in Figure \ref{fig_indK4} will be  a minor of $G$,  a contradiction since all  of them have a $H_3$ subgraph.
On the other hand, if $|V(H)|\ge6$,  then  one of the paths from $1$ to $a$, from $a$ to $3$ (within $P_{13}$),  or $P_{24}$ is subdivided.  This implies that  $G$ contains a homeomorph of $K_4$ with corner nodes $a,1,2,4$ or $a,2,3,4,$ which thus contains  two adjacent subdivided edges, giving a $H_3$ minor. 
\end{proof}
\begin{figure}[h!]
\bc  \includegraphics[scale=0.75]{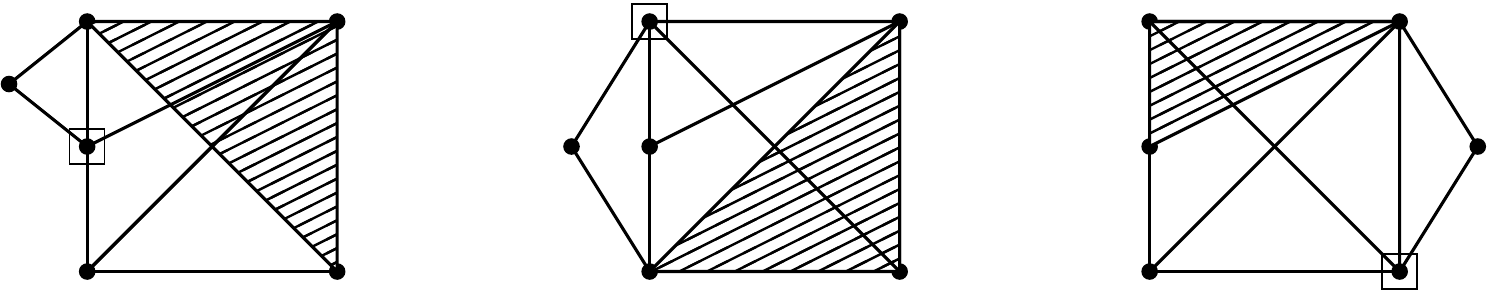}
\caption{Bad subgraphs in the  proof of Lemma \ref{lem_K4ind}.}\label{fig_indK4}
\ec
\end{figure}

Lemmas \ref{lem2sides} and \ref{lem_K4ind} imply directly:

\begin{corollary}\label{cor_noF6G6_struct}
Let $G\ne K_{3,3}$ be a 2-connected graph in $\mathcal{F}(H_3)$ on  $n\ge 6$ nodes and let $H$ be a homeomorph of $K_4$ contained in $G$.
 Then
the endnodes of at least one of the two sides  of $H$ form a cutset in $G$.
Moreover, if  $P_{13}$ is a side of $H$ and  its endnodes  $\{1,3\}$ do not form a cutset, then $P_{13}=\{1,3\}$ and there is no component of $G\setminus H$ which is connected to $P_{13}$. 
\end{corollary}


We now show that one may add edges to $G$ so that all minimal homeomorphs of $K_4$ are 4-cliques, without creating  a $F_3$ or $H_3$ minor.

\begin{lemma}\label{lem_missing_edge}
Let  $G\ne K_{3,3}$ be  a 2-connected  graph in $\mathcal{F}(F_3,H_3)$ on $n\ge 6$ nodes and let  $H$ be a homeomorph of $K_4$ contained in $G$.
The graph obtained by adding to $G$ the edges between the endpoints of the sides of  $H$ belongs to $\mathcal{F}(F_3,H_3)$.
\end{lemma}
\begin{proof}Say $P_{13}$ and $P_{24}$ are the sides of $H$.  Assume  $|V(P_{13})|\ge 3$ and $\{1,3\}\not\in E(G)$.  
By Corollary \ref{cor_noF6G6_struct},  $\{1,3\}$ is a cutset in $G$. We  show that $\hG=G+\{1,3\}\in\mathcal{F}(F_3,H_3)$. 
First,  applying Lemma \ref{lem_minor_cut} (ii) with $M=H_3$
 and $\{x,y\}=\{1,3\}$, we obtain that $\hG\in\mathcal{F}(H_3)$.

\begin{figure}[!h]
\bc  \includegraphics[scale=0.45]{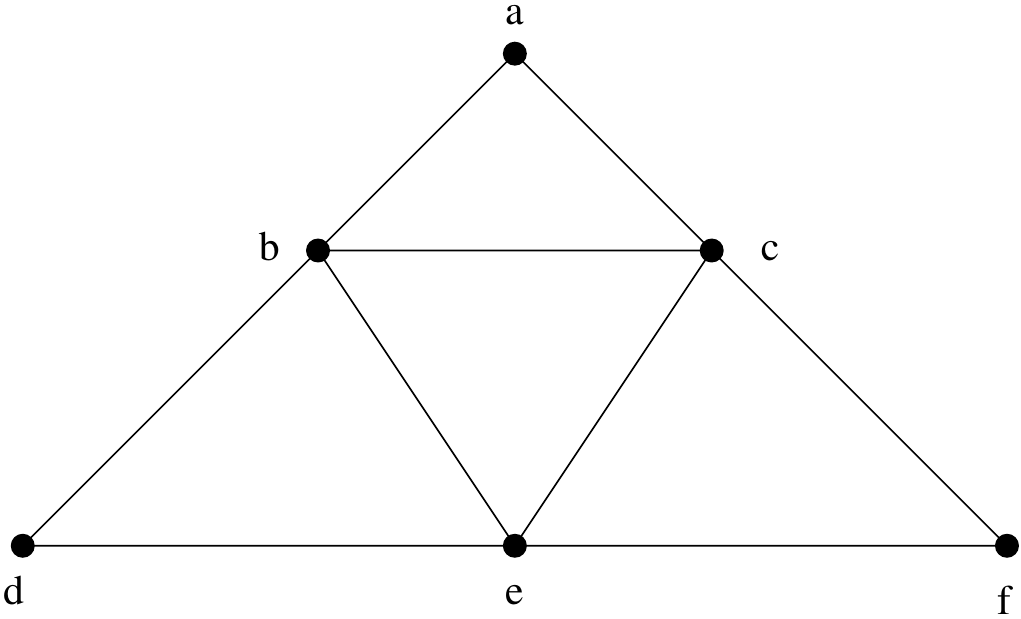}
 \caption{A labelling of $F_3$ used in Lemma \ref{lem_missing_edge}.}
 \label{fwfwe}
\ec
 \end{figure}

Next,  assume for contradiction that $\hG$ has a $F_3$ minor, where the  labelling  of $F_3$ is given  in Figure \ref{fwfwe}.    Applying Lemma \ref{lem_minor_cut} (i) with $M=F_3$  
and  $\{x,y\}=\{1,3\}$,  we see  that the nodes $1$ and $3$ belong to   distinct classes  of  the $F_3$-partition, 
which   corresponds to a cutset of $F_3$. Say, $1\in V_e$ and $3\in V_c$.
Then the nodes 2 and 4 do not lie in $V_e\cup V_c$ 
(for otherwise, one would have an $F_3$-partition in $G$).
 Next we show that the nodes $2$ and $4$ do not belong to  the same  class of the $F_3$-partition. Assume for contradiction  that $2,4\in V_k$.
 If $\{2,4\}$ is not a cutset in $G$ then, by Corollary \ref{cor_noF6G6_struct},
  $P_{24}=\{2,4\}$ and no component of $G\setminus H$ connects to $\{2,4\}$. 
 Hence $V_k=\{2,4\}$ and  we can move node $2$ to the class $V_e$, so that we obtain a $F_3$-partition of $G$, a contradiction.
 If $\{2,4\}$ is a cutset of $G$, then  every component of $G\setminus\{2,4\}$ except the one containing $1$ and $3$ has to lie within $V_k$,    so we can again move  node $2$ to $V_e$ and obtain a $F_3$-partition of $G$. 

Accordingly, the nodes $1,2,3$ and $4$ belong to  distinct classes and we can  assume without loss of generality that $2\notin V_f$. Observe that 
every $1-2$ path in $G$ is either the edge $\{1,2\}$ or meets the  nodes 3  or 4.   Similarly, every $2-3$ path in $G$ is either the edge $\{2,3\}$ or meets  the nodes 1 or 4.  An easy case analysis shows that whatever the position of nodes  2 and 4 in the $F_3$-partition we  always find a $1-2$ or a $2-3$ path violating the above conditions. 
\end{proof}

We are now ready to   show the main result of this section.

\begin{theorem}\label{theonoF3H3}
Let  $G$ be  a 2-connected  graph  with   $n\ge 6$ nodes and $G\ne K_{3,3}$. If $G \in \mathcal{F}(F_3,H_3)$ then there exists a chordal graph $Q\in \FF(F_3,H_3)$ containing $G$ as a subgraph.
\end{theorem}

\begin{proof}If $G\in\mathcal{F}(F_3,K_4)$ then we are done  by Theorem \ref{theonoF3K4}. Otherwise, we augment the graph $G$ by adding the edges between the endpoints of the sides of every homeomorph of $K_4$ contained in $G$. Let $\hG$ be the graph obtained in this way. By Lemma~\ref{lem_missing_edge},  we know  that $\hG\in\mathcal{F}(F_3,H_3)$. Hence, for each  $K_4$-homeomorph $H$ in $\hG$, its  corners  form a 4-clique. Moreover, if $C,C'$ are two distinct 4-cliques of $\hG$, then $C\cap C'$ is contained in a side of $C$ and $C'$. 

Consider a 4-clique $C=\{1,2,3,4\}$ in $\hG$, say with sides $\{1,3\}$, $\{2,4\}$ (so  each component of $\hG\setminus C$ connects to $\{1,3\}$ or to $\{2,4\}$, by Lemma \ref{lem2sides}). 
Pick an edge $f$ between the two sides (i.e., $f=\{i,j\}$ with $i\in \{1,3\}$, $j\in \{2,4\}$) and delete this edge $f$ from $\hG$.
We repeat this process  with  every 4-clique  in $\hG$ and  obtain  the graph $G_0=\hG\setminus\{f_1,\dots,f_k\}$, if $\hG$ has $k$ 4-cliques.

By construction, $G_0$ belongs to $\mathcal{F}(F_3,K_4)$ and  is 2-connected.
Hence, we can apply Theorem \ref{theonoF3K4} to $G_0$ and obtain a chordal  graph $Q_0\in \FF(F_3,K_4)$ containing $G_0$ as a subgraph. 
Hence, $Q_0$ is a clique 2-sum of free triangles. 
It suffices now  to show that the augmented graph $Q= Q_0+\{f_i:  i\in [k]\}$ is a clique 2-sum of free $K_3$'s and $K_4$'s. 
Then   $Q$ is a chordal graph in  $\FF(F_3,H_3)$ (by Theorem \ref{theofree}) containing $\hG$ and thus $G$, and the proof is completed.
 
For this, consider again a 4-clique $C=\{1,2,3,4\}$ in $\hG$ with sides $\{1,3\}$ and $\{2,4\}$.
Then, each component of $\hG\setminus C$ connects to $\{1,3\}$ or $\{2,4\}$. We claim that the same holds  for each component of $Q_0\setminus C$.
 Indeed, 
a component of $Q_0\setminus C$ is a union of some components of $\hG\setminus C$. Thus it connects to two nodes  (to 1,3, or to 2,4),
or to at least three nodes of $C$. But the latter case cannot occur since we would then find a $K_4$ minor in $Q_0$.

Assume that the edge $f=\{1,4\}$ was deleted from the 4-clique $C$ when making the graph $G_0$. We now show that adding it back to $Q_0$ 
 results in a free graph. Indeed, by adding the edge $\{1,4\}$ we only replace the two maximal 3-cliques $\{1,3,4 \}$ and $\{1,2,4\}$ by a new maximal 4-clique $\{1,2,3,4\}$, which is free.   We iterate this process for each of the edges $f_1,\ldots,f_k$  and obtain    that $Q= Q_0+\{f_i: i\in [k]\}$ is  the  clique 2-sum of free $K_3$'s and $K_4$'s. Summarizing, $Q$ is a  2-connected chordal graph  with   $\omega(G)\le 4$ which is free. Then,  Theorem \ref{theofree} $(iii)$ implies that $Q$ does not have $F_3$ or $H_3$ as minors. 
\end{proof}

\section{Characterization of graphs with $\sla(G)\le 2$}\label{seclast}


Recall that the  {\em  largeur d'arborescence}  of a graph $G$, denoted by $\lda(G)$, is defined as the smallest integer $k\ge 1$ such that $G$ is a minor of $T\square K_k$, for some tree $T$. Colin de Verdi\`ere \cite{V98} introduced the  largeur d'arborescence  as upper bound for his graph parameter $\nu(\cdot)$, which is defined as the 
maximum corank of a matrix $A\in \SSS_+^n$ satisfying the condition: $A_{ij}=0$ if and only if $i\ne j$ and $\{i,j\}\not\in E(G)$, as well as the following condition known as the {\em Strong Arnold Property}:
if $AX=0$ where $X\in \SSS^n$ satisfies $X_{ij}=0$ for all $i=j$ and all $\{i,j\}\in E(G)$,  then $X=0$. 

In \cite{V98} it was shown  that  $\nu(\cdot)$ is minor monotone and  that for any graph $G$,   $\nu(G)\le \lda(G)$. Moreover, this holds  with equality for the family of graphs~$G_r$, i.e., $\nu(G_r)=\lda(G_r)=r$ for all $r\ge 2$ (recall  Section \ref{sec_forbminor}). Furthermore, 
\begin{equation}\label{thasemamiso}
\lda(G)\le 1\Longleftrightarrow \nu(G)\le 1\Longleftrightarrow G \text{ has no minor } K_3.
\end{equation}
Lastly, Kotlov \cite{Ko00} shows:
\begin{equation}\label{wefwoerwer}
\lda(G)\le 2 \Longleftrightarrow \nu(G)\le 2 \Longleftrightarrow G \text{ has no minors } F_3, K_4.
\end{equation}
The most work in obtaining the characterization   \eqref{wefwoerwer} is  to show that  $\lda(G)\le 2$ if $G\in \FF(K_4,F_3)$. In fact, this also follows from our characterization of the class $\FF(K_4,F_3)$. 
Indeed, if $G\in \FF(K_4,F_3)$ is 2-connected then we have shown that $G$ is subgraph of $G'$ which is a clique 2-sum of free triangles. Now our argument in the 
proof of Theorem \ref{theofree}
also shows that $G'$ is a contraction minor of $T\Box K_2$ for some tree $T$ (as each triangle of $G'$ arises as  contraction of a 4-clique which can be replaced by a 4-circuit).
In this sense our characterization  is a refinement of Kotlov's result tailored to  our needs.

We now characterize the graphs with $\sla(G)\le 2$. The {\em wheel} $W_5$ is obtained from the circuit $C_4$  by adding a  node adjacent to all nodes of $C_4$. \index{graph!wheel}

\begin{theorem}
For a graph $G$,  $\sla(G)\le 2$ if and only if $G\in \FF(F_3,H_3,W_5)$.
\end{theorem}

\begin{proof}We already know that $\sla(G)\ge \egd(G)=3$ for $G=F_3,H_3$.
Suppose for contradiction that $\sla(W_5)\le 2.$
Then $\sla(W_5)=\sla(H)$ where $H$ is  a chordal extension of $W_5$ and  $H$ is a contraction minor of some $T\stp K_2$. As $W_5$ is not chordal, $H$ contains $W_5$ with one added chord on its 4-circuit, i.e., $H$ contains $K_5\setminus e$ and thus $\sla(H)\ge \sla(K_5\setminus e)=3$.
Therefore, $F_3,H_3,W_5$ are forbidden minors for the property $\sla(G)\le 2$.
Conversely, assume that $G\in\FF(F_3,H_3,W_5)$ is 2-connected, we show that $\sla (G)\le 2$. This is clear if $G$ has $n\le 4$ nodes, 
 or if $G$ has $n=5$ nodes and it has a node of degree 2.
If $G$ has $n=5$ nodes and each node has degree at least 3, then one can easily verify that $G$ contains $W_5$.
 If $G$ has $n\ge 6$ nodes  then $\sla(G)\le 2$ 
 follows from Theorem~\ref{theomain} (since  $G\ne K_{3,3}$ as $W_5\preceq K_{3,3}$).
\end{proof}

%

 Summarizing, the following inequalities are known:
     $$\nu(G)\le \lda(G) \text{ and } \egd(G)\le \sla(G)\le \lda(G).$$
      Moreover, by combining \eqref{thasemamiso}, \eqref{wefwoerwer} and Theorem \ref{thm:main} it follows that if  $G$ is a graph with $\nu(G)\le 2$, then 
$\egd(G)=\nu(G)$. Furthermore, it is known that   $\nu(K_n)=n-1$ \cite{V98} and thus 
$\nu(K_n) > \sla(K_n)\ge\egd(K_n)$ if $n\ge 4$.

An interesting open question is whether the inequality $\egd(G)\le \nu(G)$ holds in general. We point out that the analogous inequality $\nu^=(G)\le \gd(G)$ was  shown in \cite{LV12}. Recall that the parameter  $\nu^=(\cdot)$  is the analogue of $\nu(\cdot)$ studied by van der Holst \cite{H03} (same definition as $\nu(G)$, but now requiring only that $A_{ij}=0$ for $\{i,j\}\in E(G)$ and allowing zero entries at positions on the diagonal and at edges), and $\nu^=(\cdot)$ satisfies:
$\nu(G)\le \nu^=(G)$.

\bigskip
\noindent {\bf Acknowledgements.} We thank two anonymous referees for their useful comments which  helped us to improve the presentation of the paper.

\bibliographystyle{plain}

\end{document}